\documentclass[12pt,leqno,twoside]{amsart}
\usepackage{amsmath,amscd,amssymb,amsfonts,latexsym,mathrsfs,bm,extarrows,graphicx,stackrel}
\usepackage[all,cmtip]{xy}
\usepackage{ulem}
\usepackage{tikz}

\usepackage{srcltx}
\usepackage[colorlinks=true, pdfstartview=FitV, linkcolor=blue, citecolor=blue, urlcolor=blue]{hyperref}
\usepackage{subcaption} 

\usepackage{amsbsy}
\usepackage{amsthm,alltt,dsfont}

\usepackage[margin=1.28in]{geometry}

\usepackage{mathptmx}

\usepackage{amsmath, amscd,amsfonts,latexsym,mathrsfs,wasysym,bm,mathtools,extarrows,graphicx,stackrel}

\theoremstyle{plain}
\newtheorem{theorem}{Theorem}[section]
\newtheorem{corollary}[theorem]{Corollary}
\newtheorem{conjecture}[theorem]{Conjecture}
\newtheorem{lem}[theorem]{Lemma}
\newtheorem{prop}[theorem]{Proposition}

\newtheorem{problem}[theorem]{Problem}

\theoremstyle{definition}
\newtheorem{df}[theorem]{Definition}
\newtheorem{remark}[theorem]{Remark}
\newtheorem{example}[theorem]{Example}
\newtheorem{exercise}[theorem]{Exercise}

\def\be{\begin{equation}}
\def\ee{\end{equation}}

\def\bt{\begin{theorem}}
\def\et{\end{theorem}}

\def\bc{\begin{corollary}}
\def\ec{\end{corollary}}

\def\br{\begin{remark}}
\def\er{\end{remark}}

\def\bp{\begin{prop}}
\def\ep{\end{prop}}

\def\bpr{\begin{problem}}
\def\epr{\end{problem}}

\def\bl{\begin{lem}}
\def\el{\end{lem}}

\def\bn{\begin{enumerate}}
\def\en{\end{enumerate}}

\def\bex{\begin{example}}
\def\eex{\end{example}}

\def\bd{\begin{df}}
\def\ed{\end{df}}

\def\bx{\begin{exercise}}
\def\ex{\end{exercise}}

\def\ben{\begin{enumerate}}
\def\een{\end{enumerate}}

\newcommand\sL{\mathcal{L}}

\newcommand{\lk}{{\mathbb C \rm{lk}}}

\newcommand{\un}{\underline}

\def\pp{{\bf p}}
\def\uu{{\bf u}}

\newcommand\bR{{\mathbb{R}}}
\newcommand\bC{{\mathbb{C}}}

\DeclareMathOperator{\pED}{pEDdeg}

\DeclareMathOperator{\conv}{conv}

\DeclareMathOperator{\codim}{codim}              
\DeclareMathOperator{\reg}{reg}                  
                  
\DeclareMathOperator{\id}{id}                    

\DeclareMathOperator{\Eu}{Eu}

\DeclareMathOperator{\LOdeg}{LOdeg}
\DeclareMathOperator{\MLdeg}{MLdeg}
\DeclareMathOperator{\EDdeg}{EDdeg}

\DeclareMathOperator{\EDdef}{EDdefect}

\DeclareMathOperator{\Sing}{Sing}

\def\bC{\mathbb{C}}
\def\RR{\mathbb{R}}
\def\cM{\mathcal{M}}

\def\bP{\mathbb{P}}

\def\cO{\mathcal{O}}
\def\lra{\longrightarrow}
\def\bQ{\mathbb{Q}}

\def\bZ{\mathbb{Z}}

\def\C{\mathbb{C}}

\newcommand{\linearf}{f} 
\newcommand{\Hf}{H} 

\newcommand{\Hone}{H_P^{(1)}}
\newcommand{\Htwo}{H_P^{(2)}} 

\newcommand{\Hd}{H_P^{(d)}}
\newcommand{\Hk}{H_P^{(k)}}
\newcommand{\HkMinusOne}{H_P^{(k-1)}}

\newcommand{\T}{T}  			
\newcommand{\CStar}{{\mathbb{C}^*}}
\newcommand{\CStarN}{\left({\CStar}\right)^N}
\newcommand{\CStarNplusOne}{\left({\CStar}\right)^{N+1}}

\newcommand{\CC}{\mathbb{C}}
\newcommand{\lagObjective}{\Lambda}
\DeclareMathOperator{\PD}{PD}

\newcommand{\leftrarrows}{\mathrel{\raise.75ex\hbox{\oalign{%
  $\scriptstyle\leftarrow$\cr
  \vrule width0pt height.5ex$\hfil\scriptstyle\relbar$\cr}}}}
\newcommand{\lrightarrows}{\mathrel{\raise.75ex\hbox{\oalign{%
  $\scriptstyle\relbar$\hfil\cr
  $\scriptstyle\vrule width0pt height.5ex\smash\rightarrow$\cr}}}}
\newcommand{\Rrelbar}{\mathrel{\raise.75ex\hbox{\oalign{%
  $\scriptstyle\relbar$\cr
  \vrule width0pt height.5ex$\scriptstyle\relbar$}}}}

\makeatletter
\def\leftrightarrowsfill@{\arrowfill@\leftrarrows\Rrelbar\lrightarrows}
\newcommand{\xleftrightarrows}[2][]{\ext@arrow 3399\leftrightarrowsfill@{#1}{#2}}
\makeatother

\title[Applied algebraic geometry and algebraic statistics]{Applications of singularity theory \\ in applied algebraic geometry and algebraic statistics}

\author[L. Maxim]{Lauren\c{t}iu G. Maxim}
\address{L. Maxim: Department of Mathematics, University of Wisconsin-Madison, USA.}
\email {maxim@math.wisc.edu}

\author[J. I. Rodriguez]{Jose Israel Rodriguez}
\address{J. I. Rodriguez: Department of Mathematics, University of Wisconsin-Madison, USA.}
\email {jose@math.wisc.edu}

\author[B. Wang]{Botong Wang}
\address{B. Wang: Department of Mathematics, University of Wisconsin-Madison, USA.}
\email {wang@math.wisc.edu}

\begin{document}

\maketitle


\begin{abstract} We survey recent applications of topology and singularity theory in the study of the algebraic complexity of concrete optimization problems in applied algebraic geometry and algebraic statistics.\end{abstract}



\tableofcontents




\section{Introduction}

This paper surveys recent developments in the study of the algebraic complexity of concrete optimization problems in applied algebraic geometry and algebraic statistics. We will focus here on our own work on the \index{Euclidean distance degree} {\it Euclidean distance (ED) degree}, which is an algebraic measure of the complexity of nearest point problems, as well as on the \index{maximum likelihood degree} {\it maximum likelihood (ML) degree}, which measures the algebraic complexity of the maximum likelihood estimation. 
For complete details, the interested reader may consult \cite{MRW, MRW2, MRW3, MRW4, MRW5}. 

\smallskip

Without being particularly heavy on technical details, it is our hope that the results and techniques described in this note are of equal interest for pure mathematicians and applied scientists: besides acquainting applied scientists with a variety of tools from topology, algebraic geometry and singularity theory, the interdisciplinary nature of the work presented here should lead pure mathematicians to become more acquainted with a myriad of tools used in more applied research fields, such as computer vision, semidefinite programming, phylogenetics, etc.

\smallskip

We begin our introduction with a brief synopsis of \index{optimization} {\it optimization}.
Given {\index{data point} a {\it data point} $\un u \in \bR^n$ and an {\it objective function} $f_{\un u}:\bR^n \to \bR$ depending on $\un u$, a constrained \index{optimization problem} {\it optimization problem} has the form
\begin{center}
$\min / \max \ f_{\un u}(\un x)$
\end{center}
subject to polynomial constraints
\begin{center}
$g_1(\un x)=\cdots = g_k(\un x)=0$.
\end{center} 
In other words,  one aims to optimize the function $f_{\un u}$ over the real algebraic variety $$X:=V(g_1,\ldots,g_k),$$ which oftentimes is a {\it statistical model}.
To find the optimal solution, one first finds the \index{critical point} {\it critical points} of $f_{\un u}$ over $X$, i.e., smooth points $x \in X_{\rm reg}$ at which the gradient $\nabla f_{\un u}(x)$ is perpendicular to the tangent space $T_xX_{\rm reg}$ 
(or, more generally, find stratified critical points of $f_{\un u}$ on $X$).

In practice, one considers $f_{\un u}$ and $g_1,\ldots,g_k$, as {complex} functions, i.e., we regard $f_{\un u}$ as a complex  function defined on the complex variety (also denoted by $X$) defined by the Zariski closure of $X$ in $\bC^n$. 
For simplicity, one can further assume that $X$ is irreducible, and require $f_{\un u}$ to be holomorphic and have certain good properties (e.g., gradient-solvable). 
Then, for a general data point $\un u$, the number of {complex} critical points of $f_{\un u}$ on $X_{\rm reg}$ is finite and it is independent of $\un u$; it is called the \index{algebraic degree} {\it algebraic degree} of the given optimization problem. 
It is in fact a theorem that these notions are well defined.
The algebraic degree measures the \index{algebraic complexity}  {\it algebraic complexity} of the optimal solution of the optimization problem, and it is a good indicator of the running time needed to solve the problem exactly.

\smallskip

The main optimization problems considered in this survey paper are:
\begin{itemize}
\item[(a)] \index{nearest point problem} {\it nearest point problem (NPP) / ED optimization}: $X\subset \bR^n$ is an algebraic model (i.e., defined by polynomial equations), and 
\be\label{sed} f_{\un u}(\un x)=d_{\un u}(\un x)=\sum_{i=1}^n (x_i-u_i)^2\ee
is the \index{squared Euclidean distance} {squared Euclidean distance} from a general data point $\un u \in \bR^n$ to $X$. The corresponding algebraic degree  is called  the \index{Euclidean distance degree} {\it Euclidean distance (ED) degree of $X$}, and it is denoted by ${\rm EDdeg}(X)$. See Section \ref{NPP}.
\item[(b)] \index{maximum likelihood estimation} {\it maximum likelihood estimation (MLE)}: $X$ is a statistical model (family of probability distributions) and 
\be\label{lf} f_{\un u}(\un x)=\ell_{\un u}(\un x)=\prod_{i=1}^n p_i(\un x)^{u_i}\ee is a \index{likelihood function} {\it likelihood function} associated to the data point ${\un u} =(u_1,\ldots,u_n)$. The corresponding algebraic degree  is called  the \index{maximum likelihood degree} {\it maximum likelihood (ML) degree}. See Section \ref{MLE}.
\item[(c)]  \index{linear optimization} {\it linear optimization}: $X$ is an algebraic model and 
\be\label{lin}
\ell_{\un u}(\un x)=\sum_{i=1}^n u_i x_i \ee
is (the restriction to $X_{\rm reg}$ of) a general linear function. The corresponding algebraic degree  is called  the \index{linear optimization degree} {\it linear optimization (LO) degree}. See Section \ref{lin}.
\end{itemize}

In what follows, we devote separate sections to each of the above optimization problems, explaining the main results of our work over the last several years, along with the main constructions and ideas. 
Section \ref{NPP} deals with nearest point problems and the ED degree. We describe here a topological interpretation of the ED degree of an affine variety, and we explain how to apply it to the resolution of the multiview conjecture of \cite{DHOST}. In Section \ref{MLE} we introduce the ML degree and explain a proof of the Huh-Sturmfels involution conjecture from \cite{HS}. Section \ref{lin} details the linear optimization problem. We introduce here invariants similar to those of the MLE and explain their relation to the polar degrees from projective geometry. Finally, Section \ref{nong} presents an alternative approach to polynomial optimization problems via morsification; this is particularly useful when the data point ${\un u}$ is non-generic. Technical preliminaries from singularity theory are collected in Section \ref{prelim}, where we also include our recent formula from  \cite{MRW5} which computes Chern classes via logarithmic cotangent bundles. Overview subsections devoted to other developments in the field are included at the end of Sections \ref{NPP} and \ref{MLE}.

\medskip



\section{Preliminaries}\label{prelim}
In this section, we first recall some relevant terminology from singularity theory.
 Our aim is to help the non-expert reader to get a basic understanding of the important concepts of Whitney stratifications, constructible functions and characteristic cycles, Milnor fibers and vanishing cycles, and Chern-MacPherson classes for singular varieties. 
For more details, the reader is referred to classical references like \cite{Di, Max, MS, Sch}. Secondly, in Subsection \ref{log} we present a recent formula from \cite{MRW5} for computing Chern classes of quasi-projective varieties in terms of log geometry.

\subsection{Whitney stratification}\label{sec121}
Let $X$ be a complex algebraic variety. As is well known \cite{Wh1, Wh2}, such a variety can be endowed with a \index{Whitney stratification} {\it Whitney stratification}, that is, a (locally) finite partition $\mathscr{X}$ into non-empty, connected, locally closed nonsingular subvarieties $V$ of $X$ \index{strata} (called {\it strata}) which satisfy the following properties. 
\begin{enumerate}
\item[(a)] \index{frontier condition} {\it Frontier condition}: for any stratum $V \in \mathscr{X}$, the frontier $\partial V:={\bar V} \setminus V$ is a union of strata of $\mathscr{X}$, where ${\bar V}$ denotes the closure of $V$.
\item[(b)] \index{constructibility} {\it Constructibility}: the closure  	${\bar V}$ and the frontier $\partial V$ of any stratum $V \in \mathscr{X}$ are closed complex algebraic subspaces in $X$.
\end{enumerate}
In addition, whenever two strata $V$ and $W$ are such that $W \subseteq {\bar V}$, the pair $(W, {\bar V})$ is required to satisfy certain regularity conditions that guarantee that the variety $X$ is topologically equisingular along each stratum. 
For a recent algorithmic construction of Whitney stratification, see \cite{HN}.

\bex
A smooth complex algebraic variety $X$ is Whitney stratified with strata $V$ given by the connected components of $X$.
\eex

\bex
If $X$ is a complex algebraic variety whose singular locus is a finite set of points $s_1, \ldots, s_r$, then a Whitney stratification of $X$ can be given with strata $$\{X_{\reg}, \{s_1\}, \ldots, \{s_r\} \},$$ where $X_{\reg}$ denotes the locus of smooth points of $X$.
\eex

\bex[Whitney umbrella]\label{Wa} \index{Whitney umbrella}
Let $X$ be defined by $x^2=zy^2$ in $\bC^3$. The singular locus of $X$ is the $z$-axis, but the origin is ``more singular'' than any other point on the $z$-axis. A Whitney stratification of $X$ has strata
\[
V_1=X \setminus \{z-\text{axis}\},\quad 
V_2=\{(0,0,z) \mid z \neq 0\},\quad
V_3=\{(0,0,0)\}.
\]
\eex

\bex[Matrices with bounded rank]
Fix positive integers $r \leq s \leq t$. 
The variety of bordered-rank ($\leq r$) matrices
$$X_r:=\big\{ x=[x_{ij}] \in \bC^{s\times t} \mid {\rm rank}(x) \leq r \big\}$$
is Whitney stratified by the rank condition.
\eex


\subsection{Constructible functions and local Euler obstruction}\label{ss:constructibleFunctions}
Let $X$ be a complex algebraic variety with a Whitney stratification $\mathscr{X}$. 
A function $\alpha:X \to \bZ$ is called \index{constructible function} {\it $\mathscr{X}$-constructible} if $\alpha$ is constant along each stratum $V \in \mathscr{X}$. We say that $\alpha:X \to \bZ$ is constructible if it is $\mathscr{X}$-constructible for some Whitney stratification $\mathscr{X}$ of $X$.

For example, a constant function on $X$ (e.g., the  function $1_X$) is constructible. Moreover, if $\mathscr{X}$ is a Whitney stratification of $X$ and $V$ is a stratum in $\mathscr{X}$, the \index{indicator function} indicator function for $V$, that is 
\[
1_V:X\to\mathbb{Z},\quad 1_{V}(x)=
\begin{cases}
1 & x\in V\\
0 & \text{otherwise}
\end{cases}
\]
is $\mathscr{X}$-constructible.

The \index{Euler characteristic} Euler characteristic of a $\mathscr{X}$-constructible function $\alpha$ is the Euler characteristic of $X$ weighted by $\alpha$, that is,
\be \chi(\alpha):=\sum_{V \in \mathscr{X}} \chi(V) \cdot \alpha(V),
\ee
where $\alpha(V)$ is the (constant) value of $\alpha$ on the stratum $V \in \mathscr{X}$ and $\chi(V)$ is the \index{topological Euler characteristic} {\it topological Euler characteristic} of $V$. 

\bex By the additivity of the topological Euler characteristic in the complex algebraic context, one has
$$\chi(1_X) =\sum_{V \in \mathscr{X}} \chi(V) =\chi(X).$$
\eex

A fundamental role in singularity theory is played by the \index{local Euler obstruction} {\it local Euler obstruction} 
$$\Eu_X:X \to \bZ,$$ a constructible function defined by MacPherson in \cite{MP}, which is an essential ingredient in the definition of Chern classes for singular varieties.
The interested reader may consult, e.g., \cite{Bra} for an accessible introduction to the theory of characteristic classes for singular varieties.

The precise definition of the local Euler obstruction function is not needed here, but see, e.g., \cite[Section 4.1]{Di} or \cite{Bra} for an introduction. Let us only mention that $\Eu_X$ is constant along the strata of a fixed Whitney stratification of $X$, i.e., $\Eu_X$ is $\mathscr{X}$-constructible for any Whitney stratification $\mathscr{X}$. Moreover, if $x \in X$ is a smooth point then $\Eu_X(x)=1$, so in particular $\Eu_X=1_X$ if $X$ is nonsingular. On the other hand, $\Eu_X$ is sensitive to the presence of singularities: e.g., if $X$ is a curve, then $\Eu_X(x)$ is the multiplicity of $X$ at $x$.
 
 \bex[Nodal curve] \index{nodal curve}
Let $X$ be defined by the equation $xy=0$ in $\bC^2$. The origin $(0,0)$ is the unique singular point of $X$ and it has multiplicity $2$. A Whitney stratification of $X$ can be given with strata $V_1=X\setminus \{(0,0)\}$ and $V_2=\{(0,0)\}$.
Therefore, $\Eu_X$ takes the value $1$ on the smooth stratum $V_1$, and it takes the value  $2$ on $V_2$. 
\eex

 \bex[Whitney umbrella] \index{Whitney umbrella} 
Let $X$ be defined by $x^2=zy^2$ in $\bC^3$. A Whitney stratification of $X$ with strata $V_1$, $V_2$, $V_3$ is described in Example \ref{Wa}. The local Euler obstruction function $\Eu_X$ has values $1$, $2$ and $1$ along the strata $V_1$, $V_2$ and $V_3$, respectively (e.g., \cite[Example 4.3]{RW18} for details).
 \eex
 
 \bd The Euler characteristic $\chi(\Eu_X)$ of the local Euler obstruction function is called the \index{Euler-Mather characteristic}  {\it Euler-Mather characteristic} of $X$.
\ed

\br
As we will see later on, these Euler characteristics give a topological meaning to various algebraic degrees of optimization.
\er
 
We denote by $CF_{\mathscr{X}}(X)$ the abelian group of $\mathscr{X}$-constructible functions on an algebraic variety $X$ with a fixed Whitney stratification $\mathscr{X}$. This is a free abelian group with basis $\{1_V \mid V \in \mathscr{X}\}$. 
We also let $CF(X)$ be the abelian group of functions $\alpha:X \to \bZ$ which are constructible with respect to some Whitney stratification of $X$. Note that $CF(X)$ can also be defined as the free abelian group generated by \index{indicator function} indicator functions $1_Z$ of closed irreducible subvarieties $Z$ of $X$. In this language, the Euler characteristic of constructible functions is the unique linear map $$\chi:CF(X) \longrightarrow \bZ$$
defined on generators by $\chi(1_Z):=\chi(Z).$

Another distinguished basis of $CF(X)$ is given by the functions $\Eu_Z$, for $Z \subset X$ a closed irreducible subvariety. Here, $\Eu_Z$ is regarded as a constructible function on $X$ by extension by $0$ on $X \setminus Z$. Similarly, the corresponding basis for $CF_{\mathscr{X}}(X)$ consists of $\{\Eu_{\bar V} \mid V \in \mathscr{X}\}$.


\subsection{Hypersurface singularities. Milnor fiber}\label{ss:HypersurfaceSingularities}
Let $X$ be a complex algebraic variety, and let $f\colon X \to \bC$ be a non-constant algebraic  function. Denote by $X_t:=f^{-1}(t)$ the \index{hypersurface} hypersurface in $X$ defined by the fiber of $f$ over $t \in \bC$. We restrict $f$ to a small $\delta$-tube $T(X_0)$ around $X_0$ so that $f:T(X_0) \setminus X_0 \to D^*_\delta$ is a topologically locally trivial fibration, where $D^*_\delta$ is a punctured disc centered at $0 \in \bC$ of radius $\delta$ small enough.

For $x \in X_0=f^{-1}(0)$, let $B_\varepsilon(x)$ be an open ball of radius $\varepsilon$ in $X$, defined by using an embedding of the germ $(X,x)$ in an affine space $\bC^N$. Then 
\begin{equation}\label{eq:Fx}
F_x:=B_\varepsilon(x) \cap X_t
\end{equation}
for $0 < \vert t \vert \ll \delta \ll \varepsilon$ 
is called the \index{Milnor fiber} {\it Milnor fiber of $f$ at $x$}. It was introduced in \cite{Mil68}; see also  \cite[Chapter 10]{Max} or \cite{Sea} for a quick introduction and overview.

Assume now that $X$ is a nonsingular variety of complex dimension $n+1$.
If $x \in X_0$ is a smooth point, the Milnor fiber $F_x$ is contractible. If $x \in X_0$ is an {\it isolated} singularity, then $$F_x \simeq \bigvee_{\mu_x} S^n$$ has the homotopy type of a bouquet of $n$-dimensional spheres, whose homology classes are called \index{vanishing cycles} {\it vanishing cycles}; the number of these vanishing cycles is called the \index{Milnor number} {\it Milnor number of $f$ at $x$}, denoted by $\mu_x$, which can be computed algebraically as
$$\mu_x=\dim_{\bC} \bC\{x_0,\ldots ,x_n\}/{\left(\frac{\partial f}{\partial x_0},\ldots , \frac{\partial f}{\partial x_n}\right)},$$
where $\bC\{x_0,\ldots ,x_n\}$ is the $\bC$-algebra of analytic function germs defined at $x$ (with respect to a choice of coordinate functions in an analytic neighborhood of $x$). More generally, the Milnor fiber at a point in a stratum $V$ of a Whitney stratification of $X_0$ has the homotopy type of a finite CW complex of real dimension $n-\dim V$.

Let us also note here that it follows from Thom's second isotopy lemma (e.g., see \cite{Mat}) that the topological type of Milnor fibers is constant along the strata of a Whitney stratification $\mathscr{X}$ of $X_0$. For this reason, we will often denote by $F_V$ the Milnor fiber of $f$ at some point in the stratum $V \in \mathscr{X}$. 

\bex[Whitney umbrella] \index{Whitney umbrella} 
Consider the complex hypersurface $X_0=f^{-1}(0) \subset \bC^3$ defined by the polynomial  $f(x,y,z)=x^2-zy^2$.
A Whitney stratification of $X_0$ was given in Example \ref{Wa}, with strata  
\[
V_1=X \setminus \{z-\text{axis}\},\quad 
V_2=\{(0,0,z) \mid z \neq 0\},\quad
V_3=\{(0,0,0)\}.
\]
The Milnor fiber at any point in $V_1$ is contractible, the Milnor fiber at any point in $V_2$ is homotopy equivalent to a circle $S^1$, and the Milnor fiber at the point $V_3$ (the origin) is homotopy equivalent to a $2$-sphere $S^2$; e.g., see \cite[Chapter 10]{Max} for details.
\eex


\subsection{Nearby and vanishing cycle functors}\label{neva}
The fact that the topological type of Milnor fibers is constant along the strata of a Whitney stratification allows us to encode the (reduced) Euler characteristics of Milnor fibers in a constructible function. In this subsection, we focus on two functors associated to $f$ which are defined around this idea. While there is also a sheaf theoretical counterpart of these functors (e.g., see \cite{Di, Max, Sch}), we only need here their interpretation in terms of constructible functions.

Let $f\colon X \to \bC$ be a non-constant algebraic function defined on a complex algebraic variety $X$, with  $X_0:=f^{-1}(0)$. 

The \index{nearby cycle functor} {\it nearby cycle functor} of $f$,
$$\psi_f\colon CF(X) \to CF(X_0)$$
is defined as follows. For $\alpha \in CF(X)$, $\psi_f(\alpha)$ is the constructible function on $X_0$ whose value at $x \in X_0$ is given by
$$\psi_f(\alpha)(x):=\chi(\alpha \cdot 1_{F_x}),$$
where $F_x$ denotes the Milnor fiber of $f$ at $x$ (cf. Section~\ref{ss:HypersurfaceSingularities}), and ``$\cdot$'' stands for the multiplication of constructible functions.
In particular,
$\psi_f(1_X)$ is the constructible function on $X_0$ whose value at $x \in X_0$ is given by the Euler characteristic $\chi(F_x)$ of the Milnor fiber $F_x$ at $x$. 

The \index{vanishing cycle functor} {\it vanishing cycle functor} of $f$ is  defined as
\[
\varphi_f\colon CF(X) \to CF(X_0),\quad 
\alpha\mapsto \varphi_f(\alpha):=\psi_f(\alpha) - \alpha\vert_{X_0}.
\]
In particular,
$$\varphi_f(1_X)=\psi_f(1_X)-1_{X_0} \in CF(X_0)$$
is the constructible function whose value on a stratum $V$ of $X_0$ is given by the Euler characteristic of the reduced cohomology of the Milnor fiber $F_V$ at some point in $V$, i.e., 
$$\varphi_f(1_X)\vert_V=\chi(\widetilde{H}^*(F_V;\bQ)).$$

If $X$ is smooth, the fact that Milnor fibers at smooth points of $X_0$ are contractible implies that the constructible function $\varphi_f(1_X)$ is in this case supported on the singular locus of $X_0$. More generally, if $X$ is endowed with a Whitney stratification $\mathscr{X}$ and $\alpha \in CF_{\mathscr{X}}(X)$, then $\varphi_f(\alpha)$ is supported on $X_0 \cap {\rm Sing}_{\mathscr{X}}(f)$, with $${\rm Sing}_{\mathscr{X}}(f):=\bigcup_{V \in \mathscr{X}} {\rm Sing}(f\vert_V)$$
denoting the \index{stratified singular locus} {\it stratified singular locus} of $f$ with respect to $\mathscr{X}$.


\subsection{Conormal varieties. Characteristic cycles}
Let $X$ be a smooth complex algebraic variety, and denote as before by $CF(X)$ the group of (algebraically) constructible functions on $X$, i.e., the free abelian group generated by indicator functions $1_Z$ of closed irreducible subvarieties $Z$ of $X$.  

Let $L(X)$ be the free abelian group spanned by the irreducible \index{conic Lagrangian cycle} conic Lagrangian cycles in the \index{cotangent bundle} cotangent bundle $T^*X$. Recall that irreducible conic Lagrangian cycles in $T^*X$ correspond to the conormal spaces $T^*_ZX$, for $Z$ a closed irreducible subvariety of $X$.  
Here, for such a closed irreducible subvariety  $Z$ of $X$ with smooth locus $Z_{\reg}$, its \index{conormal variety} conormal variety $T^*_{Z}X$ is defined as the closure in $T^*X$ of 
\[
T^*_{Z_{\reg}}X:=\{
(z,\xi)\in T^*X \mid z \in Z_{\reg}, \ \xi \in T^*_zX, \ \xi\vert_{T_zZ_{\rm reg}}=0 \}.
\]

The \index{characteristic cycle} characteristic cycle functor $CC$ establishes a group isomorphism
$$CC\colon CF(X) \lra L(X),$$ 
which, for a closed irreducible subvariety $Z$ of $X$, satisfies:
\be\label{cc} CC(\Eu_Z)=(-1)^{\dim Z}\cdot T^*_Z X.\ee
(Recall that the collection $\{\Eu_Z\}$, for $Z \subset X$ a closed irreducible subvariety, forms a basis of $CF(X)$.)


\subsection{Chern classes of singular varieties}
In \cite{MP}, MacPherson extended the notion of Chern classes to singular complex algebraic varieties by defining a natural transformation
$$c_*:CF(-) \lra A_*(-)$$
from the functor $CF(-)$ of constructible functions (with proper morphisms) to Chow (or Borel-Moore) homology $A_*(-)$, such that if $X$ is a smooth variety then $c_*(1_X)=c^*(TX) \cap [X]$. Here,  $c^*(TX)$ denotes the total cohomology Chern class of the tangent bundle $TX$, and $[X]$ is the fundamental class of $X$.

\bd
For $\varphi \in CF(X)$, we call $c_*(\varphi)\in A_*(X)$ the \index{MacPherson Chern class} {\it MacPherson Chern class} of $\varphi$. Similarly, we call
$$\check{c}_*(\varphi):= \sum_{j\geq 0}(-1)^j \cdot {c}_{j}(\varphi),$$
the \index{signed MacPherson Chern class} {\it signed MacPherson Chern class} of $\varphi$, 
with ${c}_{j}(\varphi)\in A_j(X)$ denoting the $j$-th component of $c_*(\varphi)$.
\ed

For any locally closed irreducible subvariety
$Z$ of a complex algebraic variety $X$, the function $1_Z$ is constructible on $X$, and the class 
\be\label{csm} c_*^{SM}(Z):=c_*(1_Z) \in A_*(X)\ee
is usually referred to as the \index{Chern-Schwartz-MacPherson class} {\it Chern-Schwartz-MacPherson (CSM) class} of $Z$ in $X$. Similarly, 
the class \be\label{cma} c_*^{Ma}(Z):=c_*(\Eu_Z) \in A_*(X)\ee is called the \index{Chern-Mather class} {\it Chern-Mather class} of $Z$, where we regard the local Euler obstruction function $\Eu_Z$ as a constructible function on $X$ by setting the value zero on $X \setminus Z$.

Results of Ginsburg \cite{Gin} and Sabbah \cite{Sab} provided a microlocal interpretation of Chern classes, by showing that McPherson's Chern class transformation $c_*$ factors through the group of conic Lagrangian cycles in the cotangent bundle. We recall this construction below, following, e.g., \cite{AMSS}.


\subsection{Microlocal interpretation of Chern classes}

Let $E$ be a rank $r$ vector bundle on the smooth complex algebraic variety $X$. Let $\overline{E}\coloneqq \bP(E\oplus \mathbf{1})$ be the projective bundle, which is a fiberwise compactification of $E$ (with $\mathbf{1}$ denoting the trivial line bundle on $X$). Then $E$ may be identified with the open complement of $\bP(E)$ in $\overline{E}$. Let $\pi:E\to X$ and $\bar{\pi}:\overline{E} \to X$ be 
the projections, and let $\xi:=c^1(\mathcal{O}_{\overline{E}}(1))$
 be the first Chern class of the hyperplane line bundle on $\overline{E}$.
Pullback via ${\bar \pi}$ realizes $A_*(\overline{E})$ as a $A_*(X)$-module. 
An irreducible conic $d_C$-dimensional subvariety $C \subset E$
determines a ${d_C}$-dimensional cycle $\overline{C}$ in $\overline{E}$ and one can express $[\overline{C}] \in A_{d_C}(\overline{E})$ uniquely as:
\begin{equation}\label{sh}
    [\overline{C}]=\sum_{j={d_C}-r}^{d_C} \xi^{j-{d_C}+r} \cap {\bar \pi}^* c^E_j(C),
\end{equation}
for some $c^E_j(C) \in A_{j}(X)$. The classes $$c_{{d_C}-r}^E(C), \ldots, c_{d_C}^E(C)$$ defined by \eqref{sh} are called the \index{Chern classes} {\it Chern classes of $C$}. The sum \[ c_*^E(C)=\sum_{j={d_C}-r}^{d_C} c^E_j(C)\] is called the \index{shadow} {\it shadow} of $[\overline{C}]$. 
Let us note that if $C$ is supported on $E\vert_Z$ for a closed subset $i \colon Z \hookrightarrow X$, then $c_j^E(C)=i_* c_j^{E\vert_Z}(C)$; in particular, $c_j^E(C)=0$ for $j>\dim Z$.

For our applications, we will mainly work with conic Lagrangian cycles in cotangent bundles, in which case we have ${d_C}=r$.  If in this case we assume moreover that $E=X \times \bC^r$ is a trivial bundle (as in our later applications), then equation \eqref{sh} translates into 
\be\label{triv}
[\overline{C}]=\sum_{j=0}^r c_j^E(C) \boxtimes [\bP^{r-j}] \in A_*(X \times \bP^r).
\ee

The use of terminology ``Chern classes of $C$"
 is justified by the following result, applied to the cotangent bundle $T^*X$ and elements of the group $L(X)$ of conic Lagrangian cycles:

\bp{\rm \cite[Proposition~3.3]{AMSS}}\label{pmc}
For any constructible function $\varphi \in CF(X)$, the Chern classes of the characteristic cycle $CC(\varphi)$ are equal to  the \index{signed MacPherson Chern class} signed MacPherson Chern classes of $\varphi$, i.e., 
\begin{equation}\label{mc}
    c_j^{T^*X}\left(CC(\varphi)\right)=(-1)^j \cdot  c_j(\varphi) \in A_j(X), \ \ j=0,\ldots, \dim(X),
\end{equation}
where  $c_j(\varphi)$ denotes the $j$-th component of MacPherson's  Chern class $c_*(\varphi)$.
\ep

If $Z\subset X$ is a closed irreducible subvariety, one gets from \eqref{cc} and \eqref{mc} the following identity:
\be\label{tc}
c_*^{T^*X}(T_{Z}^*X)=(-1)^{\dim Z}\sum_{j\geq 0}(-1)^j{c}^{Ma}_{j}(Z)=(-1)^{\dim Z} \cdot \check{c}^{Ma}_{*}(Z) \in A_*(X).
\ee

\subsection{Chern classes via logarithmic cotangent bundles}\label{log}

In this subsection, we describe a result from \cite{MRW5}, which is particularly useful for calculating the Chern-Mather classes of affine and, resp., very affine varieties. \index{very affine variety}

Let $X$ be a smooth complex algebraic variety, and let $D\subset X$ be a normal crossing divisor. Let $U:=X\setminus D$ be the complement $j:U\hookrightarrow X$ the open  inclusion. Let $\Omega_X^1(\log D)$ be the sheaf of algebraic one-forms with logarithmic poles along $D$, and denote the total space of the corresponding vector bundle by \index{logarithmic cotangent bundle} $T^*(X, D)$. Note that $T^*(X, D)$ contains $T^*U$ as an open subset. Given a conic Lagrangian cycle $\Lambda$ in $T^*U$, we denote its closure in $T^*(X, D)$ by $\overline\Lambda_{\log}$. With these notations, one has the following result.

\bt\label{thm_m}{\rm \cite[Theorem 1.1]{MRW5}}  
Let $\varphi \in CF(U)$ be any constructible function on $U$. Then 
\be\label{eq_main}
c^{T^*(X, D)}_*\Big(\overline{CC(\varphi)}_{\log}\Big)=c^{T^*X}_*\big(CC(\varphi)\big) \in A_*(X),
\ee
where, if $CC(\varphi)=\sum_{k}n_k\Lambda_k$, then $\overline{CC(\varphi)}_{\log}\coloneqq\sum_{k}n_k(\overline{\Lambda_k})_{\log}$. Here, on the right-hand side of \eqref{eq_main}, $\varphi$ is regarded as a constructible function on $X$ by extension by zero.
\et

In particular, if $\varphi = \Eu_Z$ for $Z \subset U$ an irreducible subvariety, then for $\Lambda=T_{Z}^*U$ we get from \eqref{tc} and \eqref{eq_main} that:
\be\label{8}
c^{T^*(X, D)}_*(\overline\Lambda_{\log})=(-1)^{\dim Z}\sum_{j\geq 0}(-1)^j{c}^{Ma}_{j}(Z) =(-1)^{\dim Z} \cdot \check{c}^{Ma}_{*}(Z) \in A_*(X).
\ee


\br
If $X$ is projective, 
 the corresponding degree formula in \eqref{eq_main} was proved in \cite{WZ}. Moreover, if $\varphi=1_U$, extended by $0$ to $X$, formula \eqref{eq_main} reduces in this case to a well known formula of Aluffi \cite{Al1, Al2}: 
\be
	(-1)^n \cdot c^*\left( \Omega^1_X(\log D) \right) \cap [X]= \check{c}_*(1_U)=:\check{c}_*^{SM}(U) \in A_*(X),
\ee
with $n=\dim U$, and where the right hand side denotes the signed CSM class of $U$. 
\er

\br
Let us note that if $D=\emptyset$, i.e., $U=X$, then Theorem \ref{thm_m} is a tautology, as both sides compute $\check{c}_*(\varphi)$ via Proposition \ref{pmc}.
\er

While the proof of Theorem \ref{thm_m} is too technical to be discussed in a survey, let us only mention here that it can be reduced to earlier works of Ginsburg (\cite[Theorem 3.2]{Gin}). Applications of Theorem \ref{thm_m} will be given in the subsequent sections \ref{MLE} and \ref{lin}, for computing Chern-Mather classes of (very) affine varieties (in relation to maximum likelihood estimation and, resp., liniar optimization).

\medskip

For later comparison to our work, let us also mention here that in \cite[Lemme 1.2.1]{Sab}, Sabbah obtained a different kind of formula for Chern-Mather classes, which only applies in the context of {\it closed} irreducible subvarieties of a smooth ambient variety. We formulate here the complex algebraic version of Sabbah's formula.

Let $X$ be a smooth complex algebraic variety, and $Z \subset X$ an irreducible closed subvariety.
Let $T^*_ZX$ be the conormal variety of $Z$, and 
consider its {\it projectivization} $$C(Z,X):=\bP(T^*_ZX) \subset \bP(T^*X).$$
Let $\tau:C(Z,X) \to Z$ be the restriction of the projection $\bP(T^*X)\to X$ to $C(Z,X)$.
With these notations, Sabbah proved the following formula 
\bt{\rm \cite[Lemme 1.2.1]{Sab}} \label{Sabbah} 
\be\label{sf2}
c_*^{Ma}(Z)=(-1)^{\dim X - 1 - \dim Z} c^*(TX\vert_Z) \cap \tau_*\left( c(\cO(1))^{-1} \cap  [C(Z,X)] \right) \in A_*(Z),
\ee
where $\cO(1)$ is the dual of the tautological line bundle on $\bP(T^*_ZX)$ restricted to $C(Z,X)$.
\et

\br
Let us note that Sabbah's formula is evaluated in the Chow (or Borel-Moore) homology $A_*(Z)$ of the subvariety $Z$ itself, i.e., Sabbah works with $c_*\colon CF(Z) \to A_*(Z)$ and $c_*^{Ma}(Z)=c_*(\Eu_Z)$. One can, of course, push this class forward into $A_*(X)$ under the closed embedding $Z \hookrightarrow X$ using functorial properties of MacPherson's Chern class transformation, and this is in fact the way Aluffi \cite{Al3} or Parusi\'nki-Pagacz \cite{PP} use this formula to compute Chern-Mather classes of projective varieties in $\bC \bP^n$ (evaluated in $A_*(\bC \bP^n)$). However, Sabbah's formula \eqref{sf2} doesn't work well for a (very) affine variety $Z$, in which case we will use the above Theorem \ref{thm_m} to be able to relate the conormal variety of $Z$ with the projective geometry.
\er


\section{Nearest point problems. Euclidean distance degree}\label{NPP}

Many models in data science or engineering are \index{algebraic model} {\it algebraic models} (i.e., they can be realized as real algebraic varieties $X \subset \bR^n$) 
for which one needs to solve a \index{nearest point problem} {\it nearest point problem}. Specifically, for such an algebraic model $X \subset \bR^n$ and a generic  \index{data point} {\it data point} $\underline{u}=(u_1,\ldots, u_n) \in \bR^n$, one is interested to find a nearest point $\underline{u}^* \in X_{\rm reg}$ to $\underline{u}$, i.e., a point $\underline{u}^*$ which minimizes the \index{squared Euclidean distance} squared Euclidean distance $d_{\un u}$ from the given data point $\un u \in \bR^n$. (Here, $X_{\rm reg}$ denotes the smooth locus of $X$.)

The algebraic degree of the corresponding NPP is called 
 the \index{Euclidean distance degree} {\it Euclidean distance (ED) degree of $X$}, and it is denoted by ${\rm EDdeg}(X)$. The Euclidean distance degree was introduced in \cite{DHOST} as an algebraic measure of complexity of the nearest point problem, and has since been extensively studied in areas like computer vision  \cite{AH,HL,MRW}, biology \cite{GHRS}, chemical reaction networks \cite{AdHe}, engineering \cite{CNAS,SHA}, numerical algebraic geometry \cite{Hau, MR}, data science \cite{HW}, etc.

\subsection{Classical examples of nearest point problems}
Let us briefly indicate two main examples of nearest point problems. The interested reader may consult, e.g., \cite[Section 3]{DHOST} and the references therein for more such examples.
\bex[Low-rank approximation]\label{EY} \index{low-rank approximation}
Fix positive integers $r \leq s \leq t$ and set $n=st$. 
Consider the following {model} of bordered-rank ($\leq r$) matrices: 
$$X_r:=\big\{ X=[x_{ij}] \in \bR^{s\times t} \mid {\rm rank}(X) \leq r \big\} \subset \bR^n.$$
As generic data point, we choose a general $s \times t$ matrix $U=[u_{ij}] \in \bR^{s\times t}=\bR^n$. The nearest point problem can be solved in this case by using the \index{singular value decomposition} {\it singular value decomposition}. Indeed, the general matrix $U$ admits a product decomposition
$$U=T_1 \cdot {\rm diag}(\sigma_1,\ldots,\sigma_s) \cdot T_2,$$
where $\sigma_1>\cdots > \sigma_s$ are the {\it singular values} of the matrix $U$ (all of which can be assumed non-zero since $U$ is general), and $T_1$, $T_2$ are orthogonal matrices. Then the \index{Eckart-Young Theorem}
{\it Eckart-Young Theorem} (e.g., see \cite[Example 2.3]{DHOST}) states that the matrix of rank $\leq r$ closest to $U$ is:
$$U^*=T_1 \cdot {\rm diag}(\sigma_1,\ldots,\sigma_r, 0, \ldots, 0) \cdot T_2 \in X_r.$$
The other critical points of the squared distance function $d_U$ are given by
$$T_1 \cdot {\rm diag}(0,\ldots,0,\sigma_{i_1},0,\ldots,0,\sigma_{i_r}, 0, \ldots, 0) \cdot T_2,$$
where $\{i_1<\ldots < i_r\}$ runs over all $r$-element subsets of $\{1,\ldots, s\}$.
In particular, there are $s \choose r$ critical points of the squared distance function $d_U$, all of which are real matrices of rank exactly $r$. (Note that the regular part of $X_r$ consists exactly of rank-$r$ matrices.)
\eex

\bex[Triangulation problem in computer vision]\label{cv} 
In \index{computer vision} computer vision \cite{HZ}, \index{3D triangulation} triangulation (or 3D-reconstruction) refers to the process of reconstructing a point in the three-dimensional (3D) space from its two-dimensional (2D) projections in $m \geq 2$ cameras in general position.  The triangulation problem has  many  practical  applications,  e.g.,  in  tourism,  for  reconstructing  the  3D  structure of a tourist attraction based on a large number of online pictures \cite{AS};  in robotics,  for creating a virtual 3D space from multiple cameras mounted on an autonomous vehicle \cite{MR14}; for modeling clouds \cite{KK}; in filmmaking, for adding animation and graphics to a movie scene after everything is already shot, etc. If the 2D projections are given with infinite precision, then two cameras suffice to determine the 3D point.  In practice, however, various sources of ``noise'' (lens distortion, pixelation, etc.) lead to inaccuracies in the measured image coordinates.  The problem, then, is to find a 3D point which optimally fits the measured image points.

The algebraic model fitting the triangulation problem is the space of all possible $m$-tuples of such 2D projections with infinite precision, called the \index{multiview variety} {\it affine multiview variety} $X_m$; see \cite[Example 3.3]{DHOST} and \cite[Section 4]{MRW} for more details. The above optimization problem translates into finding a point $\underline{u}^*\in X_m$ of minimum distance to a (generic) point $\underline{u} \in \bR^{2m}$ obtained by collecting the 2D coordinates of $m$ ``noisy'' images of the given 3D point. Once $\underline{u}^*$ is obtained, a 3D point is recovered by triangulating any two of its $m$ projections. As already indicated above, in order to find such a minimizer $\underline{u}^*$ algebraically, one regards $X_m$ as a complex algebraic variety and examines all complex critical points of the squared Euclidean distance function $d_{\underline{u}}$ on $X_m$. Under the assumption that $m\geq 3$, the complex algebraic variety $X_m$ is smooth and $3$-dimensional, and one is then interested in computing the Euclidean distance degree $\EDdeg(X_m)$ of the affine multiview variety $X_m$.

An explicit conjectural formula for the Euclidean distance degree $\EDdeg(X_m)$ was proposed in \cite[Conjecture 3.4]{DHOST}, based on numerical computations from \cite{SSN} for configurations involving $m \leq 7$ cameras:
\begin{conjecture}[Multiview conjecture]\label{cedc} \index{multiview conjecture} The Euclidean distance degree of the affine multiview variety $X_m$ is given by:
\be\label{edc}
\EDdeg(X_m)=\frac{9}{2}m^3-\frac{21}{2}m^2+8m-4.
\ee
\end{conjecture}

This conjecture was  the main motivation for the introduction of the Euclidean distance degree in \cite{DHOST}.
\eex A proof of Conjecture \ref{cedc} was obtained in \cite{MRW} for $m\geq 3$ cameras in general position, by first giving a purely topological interpretation of the Euclidean distance degree of any complex affine variety as an Euler-Mather characteristic involving MacPherson's local Euler obstruction function.
This approach will be explained in Section \ref{affine} below. In Section \ref{proj}, we discuss topological formulae for the (projective) ED degree of complex projective varieties (cf. \cite{MRW2}), answering positively a conjecture of Aluffi-Harris \cite{AH}. Section \ref{defect} deals with a computation of the ED degree of a smooth projective variety $Y$ in terms of {\it generic} ED degrees associated to the singularities of a certain hypersurface on $Y$ (cf. \cite{MRW3}).

\medskip


\subsection{ED degrees of complex affine varieties. Multiview conjecture}\label{affine}
In this section we explain how to compute the Euclidean distance degree of a  complex affine variety as an Euler(-Mather) characteristic. We apply this computation to the resolution of the multiview conjecture (Conjecture \ref{cedc}).

\subsubsection{Euclidean distance degree}
Let us first recall the following definition from \cite{DHOST}:
\bd\label{def21} The \index{Euclidean distance degree} {\it Euclidean distance (ED) degree} {$\EDdeg(X)$} of an irreducible closed variety $X \subset \bC^n$ (e.g., the complexification of a real algebraic model) is the number of complex critical points of $$d_{\underline{u}}(\un{x})=\sum_{i=1}^{n} (x_i-u_i)^2 $$ on the smooth locus $X_{\rm reg}$ of $X$, for a general ${\underline{u}}=(u_1,\ldots, u_n) \in \bC^n$.
\ed

\bex
Every linear space $X$ has ED degree $1$.
\eex

\bex\label{mat}
As already discussed in Example \ref{EY}, if $X_r$ denotes the variety of $s \times t$ real matrices (with $s \leq t$) of rank at most $r$, then $\EDdeg(X_r)={s \choose r}$.
\eex

A general upper bound on the ED degree in terms of the defining polynomials of the variety can be given as follows.
\bp {\rm \cite[Proposition 2.6]{DHOST}} \
Let $X  \subset \bC^n$ be a variety of codimension $c$ that is cut out by polynomials $g_1, g_2, \ldots, g_c, \ldots, g_k$ of degrees $d_1 \geq d_2\geq \cdots \geq d_c \geq \cdots \geq d_k$. Then
\begin{equation}
\EDdeg(X) \leq d_1 d_2 \cdots d_c \cdot \sum_{i_1+i_2+\cdots+i_c \leq n-c} (d_1-1)^{i_1} (d_2-1)^{i_2}\cdots (d_c-1)^{i_c}.
\end{equation}
Equality holds when $X$ is a general complete intersection of codimension $c$ (hence $c=k$).
\ep

\begin{remark} Let us explain here the reason for the use of the term ``degree'' in Definition \ref{def21}, see \cite[Theorem 4.1]{DHOST} for complete details. For an irreducible closed variety $X \subset \bC^n$ of codimension $c$, consider the \index{ED correspondence} {\it ED correspondence} $\mathcal{E}_X$ defined as the topological closure in $\bC^n \times \bC^n$ of the set of pairs $(\un x, \un u)$ such that ${\un x} \in X_{\rm reg}$ is a critical point of $d_{\underline{u}}$. Note that $\mathcal{E}_X$ can be identified with the \index{conormal space} conormal space $T^*_X\bC^n$ of $X$ in $\bC^n$. In particular, the first projection $\pi_1:\mathcal{E}_X \to X$ is an affine vector bundle of rank $c$ over $X_{\rm reg}$, whereas for general data points $\underline{u}\in \bC^n$ the second projection $\pi_2:\mathcal{E}_X \to \bC^n$ has finite fibers $\pi_2^{-1}(\underline{u})$ of cardinality equal to $\EDdeg(X)$. 
\end{remark}


\subsubsection{Topological interpretation of ED degrees}

Our approach to studying ED degrees in \cite{MRW} makes use of Whitney stratifications and constructible functions, as introduced in Sections \ref{sec121} and \ref{ss:constructibleFunctions}.

Our main result from \cite{MRW} expresses the ED degree as an Euler characteristic and is precisely stated as follows.

\bt[\cite{MRW}]\label{EDa}
Let $X \subset \bC^n$ be an irreducible closed subvariety. Then, for general $\underline{u}=(u_0,\ldots,u_n) \in \bC^{n+1}$, we have:
\be\label{edaf} {\EDdeg(X)=(-1)^{\dim X} \cdot \chi({\rm Eu}_{X \setminus Q_{\underline{u}}})},\ee
{where $Q_{\underline{u}}=\{x\in \bC^n : \sum_{i=1}^n (x_i-u_i)^2=u_0\}$. } \newline
In particular, if $X$ is smooth (e.g., the affine multiview variety), then
\be\label{edafs} \EDdeg(X)=(-1)^{\dim X} \cdot\chi(X \setminus Q_{\underline{u}})\ee
for general $\underline{u}=(u_0,\ldots,u_n)  \in \bC^{n+1}$.
\et

\bex If $X=\bC$ is a complex line, then \eqref{edaf} yields:
$$\EDdeg(X)=-\chi(X \setminus Q_{\underline{u}})=-\left( \chi(X) - \chi(X \cap Q_{\underline{u}}) \right)=-(1-2)=1.$$
\eex

\bex
Consider the singular model given by the \index{cardiod curve} {\it cardioid curve} $X \subset \bC^2$ defined by $(x^2+y^2+x)^2 =x^2+y^2$.
This model has a unique singular point of multiplicity $2$ at the origin in $\bC^2$, and $X_{\reg}\cong \bC \bP^1 \setminus \{3 \ {\rm points}\}$.
Moreover, for generic $\un u$,  $X$ intersects  $Q_{\underline{u}}$ at $4$ smooth points. Then our topological formula \eqref{edaf} yields
$$\EDdeg(X)=-\chi({\rm Eu}_{X \setminus Q_{\un{u}}})=-(2-5)=3.$$
\eex

For the proof of Theorem \ref{EDa}, we first {\it linearize} the optimization problem by considering the closed embedding 
\begin{center} 
$i: \bC^n\hookrightarrow \bC^{n+1} \ , \ \  (x_1, \ldots, x_n)\mapsto ( x_1^2+\cdots+x_n^2, x_1, \ldots, x_n).$
\end{center}
Indeed, if $w_0, \ldots, w_{n}$ are the coordinates of $\bC^{n+1}$, then the function 
$\sum_{1\leq i\leq n}(x_i-u_i)^2-u_{0}$ on $\bC^n$ is the pullback of the (generic) linear function
$$w_{0}+\sum_{1\leq i\leq n}-2u_i w_i+ \sum_{1\leq i\leq n}u_i^2-u_{0}$$
on $\bC^{n+1}$.
The computation of the ED degree $\EDdeg(X)$ amounts now to counting the number of complex critical points of a generic linear function on the regular part of the affine variety  $i(X)\subset \bC^{n+1}$. Theorem \ref{EDa} is then a consequence of the following more general result from stratified Morse theory, e.g., see \cite{STV}, but also \cite{MRW6} as in Corollary \ref{cor_linear} below.
\bt\label{edf}{\rm \cite[Equation (2)]{STV}} \ 
Let $X\subset \bC^n$ be an irreducible closed subvariety. Let $\ell: \bC^n\to \bC$ be a general linear function, and let $H_c$ be the hyperplane in $\bC^n$ defined by the equation $\ell=c$ for a general $c\in \bC$. Then the number of critical points of $\ell |_{X_{\rm reg}}$ equals $${(-1)^{\dim_{\bC} X} \cdot \chi({\rm Eu}_{X\setminus H_{c}})}.$$
\et

When $X$ is smooth (e.g., the affine multiview variety), one can give a simpler proof of \eqref{edafs} by the following Lefschetz-type result applied to the smooth affine variety $i(X)$:
\bt\label{Le} {\rm \cite[Theorem 3.1]{MRW}} \ 
Let $X\subset \bC^n$ be a smooth closed subvariety of complex dimension $d$. Let $\ell: \bC^n\to \bC$ be a general linear function, and let $H_c$ be the hyperplane in $\bC^n$ defined by the equation $\ell=c$ for a general $c\in \bC$. Then:
\begin{itemize}
\item[(a)] $X$ is homotopy equivalent to $X \cap H_c$ with finitely many $d$-cells attached.
\item[(b)] the numbers of $d$-cells attached equals the number of critical points of $\ell|_X$.
\item[(c)] the number of critical points of $\ell|_X$ is equal to $(-1)^d \cdot \chi(X \setminus H_c)$.
\end{itemize}
\et

Theorem \ref{Le} is perhaps known to experts. Since at the time of writing  \cite{MRW} we were not aware of a suitable reference, we gave a proof of it by using Morse theory. 
In more detail, we considered real Morse functions of the form $\log |f|$, where $f$ is a {nonvanishing} holomorphic Morse function on a complex manifold. 
Such a Morse function has the following key properties:  
\begin{enumerate}
\item[(i)] The critical points of $\log |f|$ coincide with the critical points of $f$.
\item[(ii)] The index of every critical point of $\log |f|$ is equal to the complex dimension of the manifold  on which $f$ is defined. 
\end{enumerate}
However, as a real-valued Morse function, $\log|f|$ is almost never proper, so classical Morse theory does not apply. Instead, one needs to employ the non-proper Morse theory techniques developed by Palais-Smale \cite{PS}.


\subsubsection{Multiview conjecture} \index{multiview conjecture}

Formula \eqref{edafs} can be used to confirm the multiview conjecture \index{multiview conjecture} of \cite{DHOST} (Conjecture \ref{cedc}). Indeed, one has the following:
\bt[\cite{MRW}]\label{214} \ 
The ED degree of the affine multiview variety $X_m \subset \bC^{2m}$ corresponding to $m\geq 3$ cameras in general position satisfies:
$$\EDdeg(X_m) =-\chi(X_m \setminus Q_{\underline{u}})=\frac{9}{2}m^3 - \frac{21}{2}m^2+8m -4.$$
\et

The computation of $\chi(X_m \setminus Q_{\underline{u}})$ relies on topological and algebraic techniques from singularity theory and algebraic geometry, see \cite[Section 4]{MRW} for complete details. We indicate here only the key technical points. 
Even though both $X_m$ and $Q_{\underline{u}}$ are smooth in $ \bC^{2m}$ and they intersect transversally, their intersection ``at infinity'' is very singular.  
We regard the affine multiview variety $X_m$ as a Zariski open subset in its closure $Y_m$ in $(\bC\bP^2)^m$, with  divisor at infinity $Y_m \setminus X_m=D_{\infty}$. \footnote{A different compactification of $X_m$, in $\bC\bP^{2m}$, was considered in \cite{HL}, where the ED degree of the affine multiview variety $X_m$ was studied via characteristic classes. This leads to an upper bound for the Euclidean distance degree of $X_m$ given by: $\EDdeg(X_m) \leq  6m^3 - 15m^2 + 11m - 4.$} It can be easily seen that $Y_m$ is isomorphic to the blowup of $\bC\bP^3$ at $m$ points.
By using the additivity of the Euler-Poincar\'e characteristic, for the computation of $\chi(X_m \setminus Q_{\underline{u}})$ it suffices to calculate $\chi(Y_m)$, $\chi(D_{\infty})$, ${\chi(D_{\underline{u}})}$, $\chi(D_{\infty} \cap D_{\underline{u}})$, where $D_{\underline{u}}:=Y_m \cap {\overline Q_{\underline{u}}}$. The main difficulty arises in the calculation of ${\chi(D_{\underline{u}})}$, since $D_{\underline{u}}$ is an irreducible (hyper)surface in $Y_m$ with a $1$-dimensional singular locus. For the computation of Euler-Poincar\'e characteristics of complex projective hypersurfaces, we refer the reader to \cite{PP} or \cite[Section 10.4]{Max}. Theorem \ref{214} is then a direct consequence of the following formulae obtained in \cite[Theorem 4.1]{MRW}:
\begin{itemize}
\item[(a)] $\chi(Y_m)=2m+4$.
\item[(b)] $\chi(D_{\infty})=\frac{m^3}{6}-\frac{3m^2}{2}+\frac{16m}{3}$. 
\item[(c)] ${\chi(D_{\underline{u}})}=4m^3-9m^2+9m$. 
\item[(d)] $\chi(D_{\infty} \cap D_{\underline{u}})=-\frac{m^3}{3}+\frac{13m}{3}$.
\end{itemize}

\br
One can similarly define \index{line multiview variety} {\it line multiview varieties} \cite{BRST} or \index{anchored multiview variety} {\it anchored multiview varieties} \cite{RST}, and aim to compute their ED degrees. For anchored point and line multiview varieties, the corresponding ED degrees were recently computed in \cite{RST} following closely the arguments described above from \cite{MRW}.
\er


\subsection{Projective Euclidean distance degree}\label{proj}

When an algebraic model is realized as an  {\it affine cone} (i.e., it is defined by homogeneous polynomials), it is natural to consider it as a  {\it projective variety}. Such models are ubiquitous in data science, engineering and other applied fields, e.g. in (structured) low rank matrix approximation \cite{OSS}, low rank tensor approximation, formation shape control \cite{AHe}, and all across algebraic statistics \cite{DSS, Su}.  

\bex
The variety $X_r$ of $s \times t$ matrices of rank $\leq r$ is an affine cone.
\eex

\bd If $Y\subset \bC\bP^{n}$ is an irreducible complex projective variety, the \index{projective Euclidean distance degree} {\it projective Euclidean distance degree} of $Y$ is defined by 
$${\pED(Y)}:=\EDdeg(C(Y)),$$ where $C(Y)$ is the affine cone of $Y$ in $\bC^{n+1}$.\ed

The affine cone $C(Y)$ on a projective variety $Y$ has a very complicated singularity at the cone point, so the computation of $\pED(Y)$ via formula (\ref{edaf}) is in general very difficult. Instead, one aims to describe $\EDdeg (C(Y))$ in terms of the topology of the projective variety $Y$ itself. 
This problem has been addressed by Aluffi and Harris in \cite{AH} (building on preliminary results from \cite{DHOST}) in the special case when $Y$ is a smooth projective variety. The main result of Aluffi-Harris can be formulated as follows:
\bt\label{ah} {\rm \cite[Theorem 8.1]{AH}} \ 
Let $Y \subset \bC\bP^n$ be a {smooth} complex projective variety, and assume that $Y \nsubseteq Q$, 
{where $Q=\{\un{x}\in \ \bC\bP^n : x_0^2+\cdots + x_n^2=0\}$ }
is the \index{isotropic quadric} isotropic quadric in $\bC \bP^n$. Then 
\be\label{peds} \pED(Y)=(-1)^{\dim Y} \cdot \chi(Y \setminus (Q \cup H))\ee
where $H\subset \bC\bP^n$ is a general hyperplane.
\et

Theorem \ref{ah} was proved in \cite{AH} by using Chern classes for singular varieties, and it provides a generalization  of  \cite[Theorem  5.8]{DHOST}, where  it was assumed that the smooth projective variety $Y$ intersects the isotropic quadric $Q$ transversally, i.e., that $Y \cap Q$ is a smooth hypersurface in $Y$.  Aluffi and Harris also  conjectured that formula \eqref{peds}  should  admit  a  natural  generalization  to  arbitrary  (possibly  singular)  projective varieties by using the Euler-Mather characteristic defined in terms of the local Euler obstruction function.  We addressed their conjecture in \cite{MRW2}, where we proved the following result:
\bt\label{pedth} {\rm \cite[Theorem 1.3]{MRW2}} \ 
If $Y \subset \bC\bP^n$ is an irreducible complex projective variety, then 
\be\label{ped} \pED(Y)=(-1)^{\dim Y} \cdot \chi( {\rm Eu}_{Y  \setminus (Q \cup H)}),\ee
where $Q$ is the isotropic quadric and $H$ is a general hyperplane in $\bC\bP^n$. 
\et

The proof of Theorem \ref{pedth} is Morse-theoretic, and it employs ideas similar to those used to prove Theorem \ref{EDa}.

Note that in the case when $Y \subset \bC\bP^n$ is smooth, Theorem \ref{pedth} reduces to the statement of Theorem \ref{ah}. Theorem \ref{pedth} also generalizes \cite[Proposition 3.1]{AH}, where the ED degree of a possibly singular projective variety $Y \subset \bC\bP^n$ is computed under the assumption that $Y$ intersects the isotropic quadric $Q$ transversally. In this case, one actually computes what is called the \index{generic ED degree} {\it generic} ED degree of $Y$. For more results concerning generic ED degrees, see also \cite{AH, DHOST, HM, OSS}, and Section~\ref{defect}~below.

Our topological interpretation of ED degrees reduces their calculation to the problem of computing MacPherson's local Euler obstruction function and the Euler-Poincar\'e characteristics of certain smooth algebraic varieties (strata). We present such computations in the following examples.

\bex[Nodal curve] \index{nodal curve}
Let $Y=\{x_0^2x_2-x_1^2(x_1+x_2)=0\} \subset \bC\bP^2$. It has only one singular point $p=[0:0:1]$. Therefore, the local Euler obstruction function $\Eu_Y$ equals $1$ on the smooth locus $Y_{\rm reg}$ of $Y$, and $\Eu_Y(p)=2$. Note that $Y$ intersects the isotropic quadric $Q$ transversally at $6$ points, and it intersects a generic hyperplane $H$ at $3$ points. Moreover, $Y_{\rm reg}$ is isomorphic to $\bC^*$. By inclusion-exclusion, we then get that $\chi(Y_{\rm reg}\setminus (Q \cup H))=-9$. It then follows from \eqref{ped} that $\pED(Y)=(-1)\cdot [(-9)+2]=7$.
\eex

\bex[Whitney umbrella] \index{Whitney umbrella} \label{exw}
Consider the (projective) Whitney umbrella, i.e., the projective surface $Y=\{x_0^2x_1-x_2x_3^2=0\} \subset \bC\bP^3$. The singular locus of $Y$ is defined by $x_0=x_3=0$. The variety $Y$ has a Whitney stratification with strata: $S_3:=\{[0:1:0:0], [0:0:1:0]\}$, 
$S_2=\{x_0=x_3=0\} \setminus S_3$, and 
$S_1=Y\setminus \{x_0=x_3=0\}.$ 
It is well known that $\Eu_Y$ takes the values $1$, $2$ and $1$ along $S_1$, $S_2$ and $S_3$, respectively. Therefore, if we let $U:=\bC\bP^3 \setminus (Q \cup H)$ for a generic hyperplane $H \subset \bC\bP^3$ and $Q$ the isotropic quadric, then 
$$\chi(\Eu_Y\vert_U)=\chi(Y\cap U)+\chi(S_2\cap U).$$
The terms on the right-hand side of the above equality can be computed directly by using the inclusion-exclusion property of the Euler characteristic. One gets: $\chi(Y \cap U)=13$ and $\chi(S_2 \cap U)=-3$ (see \cite[Example 4.4]{MRW2} for complete details). Altogether, this yields that $\pED(Y)= \chi(\Eu_Y\vert_U)=10.$
\eex

\bex[Toric quartic surface] \index{Quartic surface} \label{extor}
Let $Y\subset \bC\bP^3$ be the surface defined by $$x_0^3x_1-x_2x_3^3=0.$$
	As in the previous example, $Y$ has a Whitney stratification with three strata: $S_3:=\{[0:1:0:0], [0:0:1:0]\}$, $S_2:=\{x_0=x_3=0\}\setminus S_3$ and $S_1=Y\setminus \{x_0=x_3=0\}$. The local Euler obstruction function takes values $1$, $3$ and $1$ along $S_1$, $S_2$ and $S_3$, respectively. In fact, the only nontrivial computation is for the local Euler obstruction function along $S_3$, and this can be done topologically as in \cite{RW18}. 
	Therefore, with $U:=\bC\bP^3 \setminus (Q \cup H)$ for a generic hyperplane $H \subset \bC\bP^3$ and $Q$ the isotropic quadric, we get
	$$
	\chi(\Eu_Y\vert_U)=\chi(Y\cap U)+2\chi(S_2\cap U).
	$$
	As in the previous example, terms on the right-hand side of the above equality can be computed directly by using the inclusion-exclusion property of the Euler characteristic, and one gets $\chi(Y\cap U)=16$ and $\chi(S_2\cap U)=-3$. 
	Hence 
	$$\pED(Y)=\chi(\Eu_Y\vert_U)=10.$$
\eex

\begin{remark}
In view of recent computations of the local Euler obstruction function for determinantal varieties \cite{GGR}, it is an interesting exercise to check that \eqref{edaf} or \eqref{ped} recovers the Euclidean distance degree of the variety of $s\times t$ matrices of rank $\leq r$,  as discussed in Example \ref{mat}.
\end{remark}


\subsection{Defect of ED degree}\label{defect}

We begin this section by noting that the projective ED degree 
$\pED(Y)$ is difficult to compute even if $Y \subset \bC\bP^n$ is smooth, since $Y$ and $Q$ may intersect non-transversally in $\bC\bP^n$. 
The idea is then to perturb the objective (i.e., squared distance) function to create a transversal intersection.
For this purpose, it is natural to introduce the following notion:
\bd The {\it $\un \lambda$-Euclidean distance (ED) degree} {$\EDdeg_{\un \lambda}(X)$} of a closed irreducible variety $X \subset \bC^n$ is the number of complex critical points of $$d^{\un \lambda}_{\underline{u}}(\un x)=\sum_{i=1}^{n} \lambda_i(x_i-u_i)^2 \ , \ \ {\un \lambda}=(\lambda_1,\ldots,\lambda_n)$$ on the smooth locus $X_{\rm reg}$ of $X$ (for general $\underline{u}\in \bC^n$).

Similarly, If $Y\subset \bC\bP^{n}$ is an irreducible complex projective variety, one defines the {\it projective $\un \lambda$-Euclidean distance degree} of $Y$ by 
$$\pED_{\un \lambda}(Y):=\EDdeg_{\un \lambda} (C(Y)),$$ where $C(Y)$ is the affine cone of $Y$ in $\bC^{n+1}$.

If $\lambda={\un 1}$, the vector with all entries $1$, we get the \index{unit ED degree} {\it (unit) ED degree}, $\EDdeg=\EDdeg_{\un{1}}$, resp., $\pED=\pED_{\un{1}}$.
If $\un \lambda$ is generic, we get the corresponding  \index{generic ED degrees} {\it generic ED degrees}.
\ed

Theorem \ref{pedth} can be easily adapted to the weighted context to obtain the following result:
\bt\label{wpe}
Let $Y \subset \bC\bP^n$ be an irreducible complex projective variety. Then 
\be\label{wped} {\pED_{\un \lambda}(Y)=(-1)^{\dim Y} \cdot \chi( {\rm Eu}_{Y  \setminus (Q_{\un \lambda} \cup H)})},\ee
where $Q_{\un \lambda}:=\{\un{x} \in \bC\bP^n \mid \lambda_0x_0^2+\cdots + \lambda_n x_n^2=0\}$ and $H$ is a general hyperplane in $\bC\bP^n$. 
In particular, if $Y$ is smooth, then 
\be \pED_{\un \lambda}(Y)=(-1)^{\dim Y} \cdot \chi(Y  \setminus (Q_{\un \lambda} \cup H)).\ee
\et

For {generic} $\un \lambda$, the quadric $Q_{\un \lambda}$ intersects $Y$ transversally in $\bC\bP^n$, and the computation of the generic projective ED degree $\pED_{\un \lambda}(Y)$ is more manageable, e.g., see \cite{DHOST, HM, AH}, etc.
This motivates the following:
\bd[Defect of ED degree]
If $Y \subset \bC\bP^n$ is an irreducible projective variety and $\un \lambda$ is generic, the \index{defect} {\it defect of Euclidean distance degree} of $Y$ is defined as:
$${\EDdef(Y)}:=\pED_{\un \lambda}(Y) - \pED(Y).$$
\ed

\bex\label{ex:quartic-toric}
The projective Whitney umbrella considered in Example \ref{exw} is transversal to the isotropic quadric $Q$, so its projective Euclidean distance degree coincides in this case with the generic Euclidean distance degree. In particular, the defect of Euclidean distance degree of the projective Whitney umbrella is trivial.

On the other hand, we have seen that the projective ED degree of the quartic surface of Example \ref{extor} equals $10$. Moreover, the generic ED degree can be computed in this case as in \cite{HS} and \cite{AH}, and it is equal to $14$. Therefore, the defect of Euclidean distance degree of the quartic surface equals $4$.
\eex

More generally, it is known that $\EDdef(Y)$ is non-negative, but for many varieties appearing in optimization, engineering, statistics, and data science, this defect is quite substantial.  In \cite{MRW3}, we gave a new topological interpretation of this defect in terms of invariants of singularities of $Y \cap Q$ (i.e., the non-transversal intersection locus) when $Y$ is a smooth irreducible complex projective variety in $\bC\bP^n$. Specifically, we proved the following result:

\bt\label{defth}{\rm \cite[Theorem 1.5]{MRW3}} \ 
Let $Y \subset \bC\bP^n$ be a smooth irreducible variety, with $Y \nsubseteq Q$, and let $Z={\rm Sing}(Y\cap Q)$. Let $\mathcal{V}$ be the collection of strata of a Whitney stratification of $Y \cap Q$ which are contained in $Z$, and choose $\un \lambda$ generic.  Then:
\be\label{deff}
{\EDdef(Y)=\sum_{V \in \mathcal{V}}  \alpha_V \cdot \pED_{\un \lambda}(\bar V)},
\ee
where, for any stratum $V\in \mathcal{V}$, $$\alpha_V=(-1)^{\codim_{Y \cap Q} V} \cdot \left( \mu_V-\sum_{\{S \mid V < S\}} \chi_c(\lk_{\overline S}(V)) \cdot \mu_S \right),$$  
with $\mu_V=	\chi(\widetilde{H}^*(F_{V};\bQ))$
the Euler characteristic of the reduced cohomology of the Milnor fiber $F_{V}$ of the hypersurface $Y \cap Q \subset Y$ at some point in $V$, and $\lk_{\overline S}(V)$ the complex link of a pair of distinct strata $(V,S)$ with $V \subset {\bar S}$.
\et

The proof of Theorem \ref{defth} relies on the theory of vanishing cycles, adapted to the pencil of quadrics $Q_{\un\lambda}$ on $Y$, see \cite[Section 2]{MRW3} for complete details. 

Note that computing the ED degree defect of $Y \subset \bC\bP^n$ yields a formula for the projective ED degree $\pED(Y)$ only in terms of generic ED degrees (which, as already mentioned, are easier to compute). Also, computing the ED degree defect directly is generally much easier than the individual computations of $\pED(Y)$ and $\pED_{\un\lambda}(Y)$ for generic $\un \lambda$. 

As an immediate consequence of Theorem \ref{defth}, one has the following result from \cite[Corollary 6.3]{AH}:

\begin{corollary}\label{cah} Under the notations of Theorem \ref{defth}, assume that  $Z={\rm Sing}(Y\cap Q)$ has only isolated singularities. Then
\be \EDdef(Y)=\sum_{x \in Z} \mu_x,\ee
where $\mu_x$ is the \index{Milnor number} Milnor number of the isolated hypersurface singularity germ $(Y\cap Q,x)$ in $Y$.
\end{corollary}

Furthermore, if $Y \cap Q$ is equisingular along the non-transversal intersection locus $Z$, then Theorem \ref{defth} yields the following:

\begin{corollary} Under the notations of Theorem \ref{defth}, assume that  $Z={\rm Sing}(Y\cap Q)$ is connected and $Y \cap Q$ is equisingular along $Z$. Then:
\be
\EDdef(Y)=\mu \cdot \pED_{\un\lambda}(Z),
\ee
where $\mu$ is the Milnor number of the isolated transversal singularity at some point  $x \in Z$ (i.e., the Milnor number of the isolated hypersurface singularity in a normal slice to $Z$ at $x$).
\end{corollary}

Since intersecting $Y$ with a general linear space $L$ does not change the multiplicities $\alpha_V$ on the right-hand side of formula \eqref{deff}, Theorem \ref{defth} also has the following immediate consequence:

\begin{corollary}
With the notations of Theorem \ref{defth}, and for $L$ a general linear subspace of $\bC\bP^n$, we have:
\be  \EDdef(Y \cap L)=\sum_{V \in \mathcal{V}}  \alpha_V \cdot \pED_{\un\lambda}(\bar V \cap L).
\ee
\end{corollary}

We conclude this section with the following example:	
\bex[$2 \times 2$ matrices of rank $1$]
Let $Y=\{x_0x_3-x_1x_2=0\} \subset \bC\bP^3$, with isotropic quadric $Q=\{\sum_{i=0}^3 x_i^2=0\}$. 
Then $Y \cap Q$ consists of $4$ lines, with $4$ isolated double point singularities (hence, each having Milnor number $1$). Corollary \ref{cah} yields that $\EDdef(Y)=4$. In fact, as shown in \cite{DHOST}, one has in this case that $\pED(Y)=2$ and $\pED_{\un\lambda}(Y)=6$ for generic $\un\lambda$. (However, the computation of both ED degrees separately is much more complicated than computing the ED defect.) 
For a higher-dimensional generalization of this example, see \cite[Example 3.3]{MRW3}.
\eex


\subsection{Other developments}

The study of Euclidean distance degrees has branched into several different areas. 
For instance, \cite{KOL} studies a different notion of distance given by the \index{$p$-norm} $p$-norm. 
This leads to counting the critical points of a structured polynomial function on a variety. Other work has gone into studying the number of real critical points of the ED function as the data varies continuously. This is done by computing the \index{ED discriminant} {\it ED discriminant}~\cite[Section 7]{DHOST} and more generally data loci~\cite{Horobet-data-loci,HR-data-loci}. 
Another point of interest is the average number of real critical points~\cite[Section 4]{DHOST} of the distance function or other objective functions. 

Focussing on a particular class of models is an other important line of work for Euclidean distance degrees. 
For example, the generic Euclidean distance degree for toric models has been studied be Helmer and Sturmfels in \cite{HM}, where they provide a combinatorial formula for the ED degree in terms of polyhedral geometry. 
However, it is still an open problem to give an analogous formula for the defect and (unit) ED degree for these models that can be applied to Example~\ref{ex:quartic-toric}.

Lastly for this section, we remark on an exciting new application.  A \index{bottleneck} {\it bottleneck} of a metric space $X$ (an algebraic model for instance) is a local minimum of the squared distance function $\mathrm{dist}^2: X\times X\to \mathbb{R}$,
$\mathrm{dist}(x,y)=  \| x-y\|^2, x\neq y $ on $X\times X$.
The smallest value of distance function on the bottlenecks 
is a fundamental invariant  in the algebraic geometry of data.
When $X$ is a real affine algebraic variety, there is a nice  necessary condition for bottlenecks \cite{DEW-bottleneck}: 
a pair $(x,y)\in X\times X$ of distinct smooth points is a bottleneck of $X$ provided that the Euclidean normal spaces at $x$ and $y$ contain the line spanned by $x$ and $y$. 
In other words, the pair of points $(x, y)$ is a critical point of the squared distance function $\mathrm{dist}^2: X\times X\to \mathbb{R}$. 
Therefore, finding lines orthogonal at two or more points is an optimization problem with algebraic constraints and 
a recent line of investigation 
was initiated in \cite{DEW-bottleneck} where a formula for the \index{bottleneck degree} bottleneck degree of a smooth variety in generic position is given in terms of of polar and Chern classes, but a topological interpretation of the bottleneck degree remains to be found.


\section{Maximum likelihood estimation}\label{MLE}
A natural problem is to describe given data in terms of a model. Maximum likelihood estimation (MLE) \index{maximum likelihood estimation} is one such approach, and it is a fundamental computational problem in statistics. For MLE, one has a likelihood function which assigns to each point in the model the likelihood of observing the given data. So by maximizing the likelihood function we gain an understanding of the data.

Consider the following example of a biased coin.~\bex
Let $\theta$ be the probability of observing tail (T) on a \index{biased coin} biased coin, and perform the following experiment: flip a biased coin twice and record the outcomes. Let
$p_i(\theta)$ be probability of observing $i$ heads (H), for $i=0,1,2$. Hence
\begin{center}
$p_0(\theta)=\theta^2$, $p_1(\theta)=2\theta(1-\theta)$, $p_2(\theta)=(1-\theta)^2$.
\end{center}
Repeat the experiment a number of times, and let 
$u_i$ record the number of times $i$ heads were observed, $i=0,1,2$.
The MLE problem is to estimate $\theta$ by maximizing the likelihood function \index{ likelihood function}
$$\ell_{\un u}(\theta)=p_0(\theta)^{u_0} p_1(\theta)^{u_1}p_2(\theta)^{u_2}.$$
For this purpose, one first solves  $d\log \ell_{\un u}=0$ for $\theta$, with unique solution $$\hat{\theta}=\frac{2u_0+u_1}{2u_0+2u_1+2u_2}.$$
Note that the distribution $p$ of this example lives in the statistical model $X=V(g)$ defined by $$g(p_0,p_1,p_2)=4p_0p_2 - p_1^2.$$ This is the \index{Hardy-Weinberg curve} {\it Hardy-Weinberg curve}, which plays an important role in population genetics (e.g., see \cite{HC}).
\eex

More generally, suppose $X \subset \Delta_n$ is a family of probability distributions, where $\Delta_n$ is the $n$-dimensional \index{probability simplex} {\it probability simplex}, i.e., 
\begin{center}$\Delta_n=\{\un p=(p_0,\ldots, p_n) \in \bR^{n+1} \mid p_i>0, \sum_i p_i=1\}.$\end{center}
Given $N$ independent and identically distributed samples, we summarize the outcome in the data vector $\un u=(u_0,\ldots, u_n)$, with $N=\sum_i u_i$ and $u_i:=$ the number of times state $i$ was observed.
Let $p_i$ be the probability of observing state $i$.
The {\it MLE / ML optimization} \index{ML optimization} consists of maximizing the likelihood function $$\ell_{\un u}(\un p):=\prod_{i=0}^n p_i^{u_i},$$ subject to the constraint $p \in X$. However, note that a parametrization of $X$ like in the above example may not be available.

The algebraic degree of ML optimization is the \index{ML degree} {\it ML degree},  denoted by $\MLdeg(X)$. It was introduced by Catanese-Hosten-Khetan-Sturmfels  \cite{CHKS} in 2006, and studied since, e.g., by Huh, Sturmfels, etc., see \cite{DSS,Hu2,HS}. The goal of this section is to describe the main ideas and constructions behind our proof in \cite{MRW5} of the Huh-Sturmfels {\it involution conjecture} of \cite{HS}.


\subsection{ML degree of very affine varieties}\label{mlva}

Recall that an affine variety $Z$ is called \index{very affine} {\it very affine}, if it admits a closed embedding to an affine torus $(\C^*)^n$ for some $n$. (For a very affine variety we always assume that such a closed embedding is chosen.) A  \index{master function} {\it master function} (also called a \index{likelihood function}  {\it likelihood function} in \cite{Hu2}) 
on $(\C^*)^n$ is of the form 
\[
\ell_{\un{u}}(\un x)\coloneqq x_1^{u_1}\cdots x_n^{u_n},
\]
where $(x_1,\ldots,x_n)$ are the coordinate functions on $(\C^*)^n$ and  $\un{u}=(u_1, \ldots, u_n)\in \bZ^n$. If, more generally, $\un{u}\in \C^n$, then $\ell_{\un{u}}$ is a multivalued function, but the critical points of $\ell_{\un{u}}|_{Z_{\reg}}$ are well defined, and they are exactly the degeneration points of the restriction of the holomorphic 1-form 
\[
d\log \ell_{\un{u}}=u_1\frac{dx_1}{x_1}+\cdots +u_n\frac{dx_n}{x_n}
\]
to $Z_{\reg}$. 

\bd\label{definition:mldegree-very-affine}
The \index{maximum likelihood degree} {\it ML degree} of a very affine variety $Z \subset (\C^*)^n$, denoted by $\MLdeg(Z)$, is the number of critical points of a likelihood/master function $\ell_{\un u}$ on  $Z_{\reg}$, for general $\un u \in \bC^n$.
\ed

The following result was obtained by Huh in \cite{Hu2}.
\bt{\rm \cite[Theorem 1]{Hu2}} \ 
If $Z \subset (\C^*)^n$ is a smooth very affine variety with $d=\dim Z$, then
$$\MLdeg(Z)=(-1)^d \cdot \chi(Z).$$
\et

Furthermore, the second and third authors generalized Huh's result to the singular setting by showing that singular strata in a Whitney stratification of $Z$ contribute to the Euler characteristic a weight given by the value of the local Euler obstruction function $\Eu_Z$ on that stratum. More precisely, one has the following.
\bt[\cite{RW}]\label{mlds} \
If $Z \subset (\C^*)^n$ is a very affine variety with $d=\dim Z$,
$$\MLdeg(Z)=(-1)^d \cdot \chi(\Eu_{Z}).$$
\et

The above formulae for the ML degree were further generalized and given a geometrical interpretation in \cite{MRW5}, as we shall now explain.

The total space of all critical points of the master functions on $Z$ defines a closed subvariety of $Z_{\reg}\times \bC^{n}$: \[
\mathfrak{X}^\circ(Z) :=\{ (\un{z}, \un{u}) \in Z_{\reg}\times \C^n\mid \un{z} \text{ is a critical point of } \ell_{\un{u}}|_{Z_{\reg}}\}.
\]
Using the natural compactifications $(\C^*)^n\subset \bP^n$ and $\C^n\subset \bP^n$, we consider $Z_{\reg}\times \C^n$ as a locally closed subvariety of $\bC\bP^n\times \bC\bP^n$. Let $\mathfrak{X}(Z)$ be the closure of $\mathfrak{X}^\circ(Z)$ in $\bC\bP^n\times \bC\bP^n$, and set 
\be\label{eq:vbidegree}  
[\mathfrak{X}(Z)]:=\sum_{i=0}^{d}v_i[\bC\bP^{i}\times \bC\bP^{n-i}]\in A_*(\bC\bP^n\times \bC\bP^{n}).
\ee
{We call $v_i$ the \index{bidegree} {\it $i$th bidegree} of $\mathfrak{X}(Z)$, and 
note that $v_0=\MLdeg(Z)$. } 

We then have the following result of \cite{MRW5}, generalizing \cite[Theorem 2]{Hu2} which covered only the case of a smooth and sch\"on very affine variety. This shows that there are deeper relations between the algebraic complexity of the MLE problem and the topology of the corresponding algebraic variety beyond the Euler characteristics considered in Theorem \ref{mlds}.

\bt{\rm \cite[Theorem 1.3]{MRW5}}\label{thm_CSM}
Let $Z\subset (\bC^*)^n$ be a very affine variety of dimension $d$, with $\mathfrak{X}(Z)$  and its corresponding bidegrees $v_i$ defined as in \eqref{eq:vbidegree}. 
Then, the \index{Chern-Mather class} Chern-Mather class of $Z$ is given by
\[
c_*^{Ma}(Z)=\sum_{i=0}^d(-1)^{d-i} v_i [\bC\bP^{i}]\in A_*(\bC\bP^n),
\]
where $c_*^{Ma}(Z):=c_*(\Eu_Z)$, with $c_*:CF(\bC \bP^n) \to A_*(\bC \bP^n)$ the MacPherson Chern class transformation and $\Eu_Z$ the local Euler obstruction function of $Z$ regarded as a constructible function on $\bC\bP^n$ by extension by zero.
\et

\begin{proof}[sketch]
Theorem \ref{thm_CSM} makes use of Theorem \ref{thm_m} applied to the reduced normal crossing divisor $D=X \setminus U=\bC\bP^n \setminus (\bC^*)^n$, i.e., $D$ is the usual boundary divisor of the toric variety $\bC\bP^n$ with open subtorus $(\bC^*)^n$, so that the log cotangent bundle $E=T^*(X,D)=\bC\bP^n \times \bC^n$ is trivialized via the global log forms $\frac{dx_1}{x_1}, \ldots, \frac{dx_n}{x_n}$. Thus, we can identify the compactification $\overline{E}$ with $\bC\bP^n\times \bC\bP^n$, with the first factor being the base and the second being the fiber. 

Given a very affine variety $Z\subset U=(\bC^*)^n$, we have by definition that 
\begin{equation}\label{eq_reg}
\mathfrak{X}^\circ(Z)=T^*_{Z_{\textrm{reg}}}(\bC^*)^n.
\end{equation}
Let $\Lambda=T^*_Z(\bC^*)^n$ be the closure of $T^*_{Z_{\textrm{reg}}}(\bC^*)^n$ in $T^*(\bC^*)^n$. Then $\Lambda$ is a conic Lagrangian cycle in $T^*(\bC^*)^n$. Let $\overline{\Lambda}_{\log}$ be the closure of $\Lambda$ in $E$ and note that $\mathfrak{X}(Z)$ is the closure of $\overline{\Lambda}_{\log}$ in $\overline{E}=\bC\bP^n\times \bC\bP^n$. Hence, by \eqref{triv}, if 
\[
c_*^{E}(\overline{\Lambda}^{\log})=\sum_{i=0}^d v_i [\bC\bP^i] \in A_*(\bC\bP^n),
\]
with $d=\dim Z$, then
\[
[\mathfrak{X}(Z)]=\sum_{i=0}^d v_i [\bC\bP^i \times \bC\bP^{n-i}] \in A_*(\bC\bP^n\times \bC\bP^n).
\]
Next note that if we consider $\Eu_Z$ as a constructible function on $\bC\bP^n$, with value equal to zero {outside $Z$}, this corresponds to the pushforward of $\Lambda$ under the open inclusion 
$j:(\bC^*)^n \hookrightarrow \bC\bP^n$. Let $\phi:T^*\bC\bP^n \to T^*(\bC\bP^n,D)$ be the natural bundle map. It then follows by Theorem \ref{thm_m} that
\begin{align*}
c_*^{E}(\overline{\Lambda}_{\log})
= c_*^{T^*\bC\bP^n}(\phi^* \overline{\Lambda}_{\log}) 
= (-1)^d \cdot c_*^{T^*\bC\bP^n}(CC(\Eu_Z)) 
= (-1)^d \cdot \check{c}_*^{Ma}(Z),
\end{align*}
or, equivalently,
\[
c_*^{Ma}(Z)=\sum_{i=0}^d (-1)^{d-i}v_i [\bC\bP^i] \in A_*(\bC\bP^n).
\]
\end{proof}

\br
Note that by taking the degrees in Theorem \ref{thm_CSM}, one recovers the statement of Theorem \ref{mlds}.
\er

\br\label{prco}
In \cite[Theorem 2]{Hu2}, the total space of critical points $\mathfrak{X}^\circ(Z)$ is defined as a subvariety of $Z\times \bC\bP^{n-1}$, and hence $\mathfrak{X}(Z)$ is a subvariety of $\bC\bP^n\times \bC\bP^{n-1}$. When $Z$ is not the ambient space $(\bC^*)^n$, our definition of $\mathfrak{X}(Z)$ is a cone of the one in \cite{Hu2}, hence in this case the two constructions define the same sequence of numbers $v_i$. In \cite{MRW5}, we used the above-mentioned construction because it gives the correct formula even when $Z$ is the ambient space $(\bC^*)^n$, but also due to our use of Chern classes of conic cycles. Note that, to understand the Chern classes of conic cycles $\Lambda$ in a vector bundle $E$, one loses track of all conic cycles supported on the zero section of the vector bundle if taking the projective cones $\bP(\Lambda)\subset \bP(E)$ instead of taking the closure $\overline{\Lambda}\subset \overline{E}=\bP(E\oplus \C)$. 
\er


\subsection{Likelihood geometry in $\bC\bP^n$}
Let $p_0,\ldots,p_n$ be the coordinates in $\bC\bP^n$ (e.g., representing probabilities).
Let $\un u=(u_0,\ldots, u_n)$ be the observed \index{data vector} data vector, where $u_i$ is the number of samples in state $i$.
The \index{likelihood function} {\it likelihood function} on $\bC\bP^n$ is given by
\begin{center}
$\ell_{\un u}(\un p)=\frac{p_0^{u_0}p_1^{u_1}\cdots p_n^{u_n}}{(p_0+\cdots +p_n)^{u_0+\cdots + u_n}}.$
\end{center}
So $\ell_{\un u}$ is a rational function on $\bC\bP^n$, regular on $\bC\bP^n \setminus \mathcal{H}$, where \begin{center}${\mathcal{H}:=\{p_0 \cdots p_n (p_0+\cdots +p_n)=0\}}$.\end{center} 
Consider the restriction of $\ell_{\un u}$ to a closed irreducible subvariety $X \subset \bC\bP^n$ (e.g., defined over $\bR$), so that  
\begin{equation}\label{eq:X-circ}
X^\circ:=X \setminus \mathcal{H} \neq \emptyset.
\end{equation}
(When $X$ is a statistical model, the ML problem is to maximize $\ell_{\un u}$ over $X \cap \Delta_n$, with $\Delta_n$ the probability simplex.) 
Let us note here that $X^\circ:=X \setminus \mathcal{H}$ is a {\it very affine variety}, in fact a closed subvariety of $(\bC^*)^{n+1}$.

\bd\label{definition:mldegree-statistics} With the above notations, the \index{maximum likelihood degree} {\it ML degree} of $X  \subset \bC\bP^n$, which will be denoted by $\MLdeg(X)$, is the number of critical points of $\ell_{\un u}$ on $X_{\rm reg} \setminus \mathcal{H}=X^\circ_{\rm reg}$. 
This is the same as the ML degree of the very affine variety $X^\circ$. 

\ed

Let us next consider the \index{likelihood correspondence} {\it likelihood correspondence} $\mathcal{L}_X$, defined as the closure in $\bC\bP^n \times \bC\bP^{n+1}$ of the set
\begin{center}
$\{(\un p, \un u) \in X^\circ_{\rm reg} \times \bC^{n+1} \mid \un p \ \text{is a critical point of} \  \ell_{\un u}\vert_{X^\circ_{\rm reg}}\} .$
\end{center}
Following \cite{HS}, we make the following definition.
\bd[ML bidegrees]\label{mlbdg}
The \index{maximum likelihood bidegree} {\it $i$-th ML bidegree}  $b_i$ of $X$, $i=0,\ldots,d=\dim X$, is given by:
$$[\mathcal{L}_X]=\sum_{i=0}^{d} {b_i} [\bC\bP^i \times \bC\bP^{n+1-i}] \in A_*(\bC\bP^n \times \bC\bP^{n+1}).$$
\ed
We note immediately that $b_0=\MLdeg(X)$ and $b_d=\deg(X)$.

\br
Note that our definition of the likelihood correspondence variety $\sL_X$ yields a subvariety of $\bC\bP^n\times \bC\bP^{n+1}$, instead of $\bC\bP^n\times \bC\bP^{n}$ as used in \cite{HS}. This is justified just in Remark \ref{prco}.
\er

In \cite{MRW5}, we proved the following result, which can be seen as a stepping stone towards proving the involution conjecture (as will be discussed in the next section).
\bt{\rm \cite[Theorem 1.6]{MRW5}}\label{mrww1}
Let $X \subset \bC\bP^n$ be a $d$-dimensional closed irreducible subvariety with $X^\circ=X \setminus \mathcal{H} \neq \emptyset$. Then the total \index{Chern-Mather class} Chern-Mather class of $X^\circ$ is:
\be\label{mrww1a} c_*^{Ma}(X^\circ)=\sum_{i=0}^d (-1)^{d-i} b_i [\bC\bP^i] \in A_*(\bC\bP^n).\ee
 Here, $c_*^{Ma}(X^\circ):=c_*(\Eu_{X^\circ})$, with $c_*:CF(\bC\bP^n)\to A_*(\bC\bP^n)$ the MacPherson-Chern class transformation, and $\Eu_{X^\circ}$ is regarded as a constructible function on $\bC\bP^n$ by extending it by $0$.
\et

The key step for proving Theorem \ref{mrww1} is formula \eqref{eq_main}, applied to the reduced normal crossing boundary divisor 
$D$ in $\bC\bP^n$,
whose irreducible components are the projective hyperplanes $D_i=\{p_i=0\} \subset \bC\bP^n$ (for $i=0,\ldots, n$), together with $D_+=\{p_+:=p_0+\cdots + p_n=0\} \subset \bP^n$. The homogeneous coordinates $[p_1,\ldots, p_n, p_+]$ are used here to identify $\bC\bP^n \setminus D_+=\bC^n$ with coordinates $x_i:=\frac{p_i}{p_+}$ ($i=1,\ldots,n$), so that 
$$\bC^n \cap D_0=\{ x_1+\cdots +x_n=1\}.$$

\br
In fact, using the embedding $\bC\bP^n$ into $\bC\bP^{n+1}$  by 
$$(p_0,\dots,p_n)\mapsto (p_0,\dots,p_n,p_+)$$
with $p_+$ defined as above, we can reduce this Theorem~\ref{mrww1}
to Theorem~\ref{thm_CSM} applied to the variety $X^\circ \in (\bC^*)^{n+1}$.
\er


\subsubsection{Sectional ML degrees and the Involution Conjecture}
Besides the ML bidegrees of Definition \ref{mlbdg}, another 
natural generalization of the ML degree is provided by the sectional ML degrees introduced in \cite{HS}. In the notations of the previous subsection, these  can be defined as follows.

\bd[Sectional ML degrees]\label{secmld}
Let $X \subset \bC\bP^n$ be a closed irreducible subvariety with $X^\circ=X \setminus \mathcal{H} \neq \emptyset$. 
The \index{sectional maximum likelihood degree} {\it $i$-th sectional ML degree} of $X$ is:
$$s_i:=\MLdeg(X \cap L_{n-i}),$$
where $L_{n-i}$ is a general linear subspace of $\bC\bP^n$ of codimension $i$.
\ed
Once again, we note that $s_0=\MLdeg(X)$ and, if $d=\dim X$, then $s_d=\deg(X)$. 

In \cite{HS}, Huh and Sturmfels conjectured that the ML bidegrees and the sectional ML degrees of a variety determine each other under some involution formulas, and proved the case when the variety $X^\circ$ is smooth and Sch\"on. The Huh-Sturmfels \index{involution conjecture} {\it Involution Conjecture} was proved in full generality in \cite{MRW5}. In what follows, we formulate our result and sketch the main ideas of its proof.

\bt{\rm \cite[Theorem 1.5]{MRW5}}\label{HSconj}
Let  $X \subset \bC\bP^n$ be a $d$-dimensional closed irreducible subvariety with $X^\circ=X \setminus \mathcal{H} \neq \emptyset$, and set
\[
B_X(\pp, \uu)=(b_0 \cdot \pp^{d}+b_1\cdot \pp^{d-1}\uu+\cdots +b_d\cdot  \uu^d)\cdot \pp^{n-d},
\]
\[ S_X(\pp, \uu)=(s_0 \cdot \pp^{d}+s_1\cdot \pp^{d-1}\uu+\cdots +s_d\cdot  \uu^d)\cdot \pp^{n-d}.\] 
Then
\be\label{st1}
B_X(\pp, \uu)=\frac{\uu\cdot S_X(\pp, \uu-\pp)-\pp\cdot S_X(\pp, 0)}{\uu-\pp},\ee
\be\label{st2}
S_X(\pp, \uu)=\frac{\uu\cdot B_X(\pp, \uu+\pp)+\pp\cdot B_X(\pp, 0)}{\uu+\pp}.
\ee
\et

The proof of Theorem \ref{HSconj} follows from the geometric interpretation of $c_*^{Ma}(X^\circ)$ given in Theorem \ref{mrww1} together with an  involution formula of Aluffi (cf. \cite{Al}, as reformulated in \cite[Corollary 2.5]{MRW5}) which we recall below.

For a constructible function $\alpha$ on $\bC\bP^n$, let $\alpha_j:=\alpha\vert_{L_{n-j}}$ be the restriction of $\alpha$ to a codimension $j$ generic linear subspace. For instance, if $\alpha=\Eu_Z$ for a locally closed subvariety $Z$ of $\bC\bP^n$, then $\alpha_j=\Eu_{Z \cap L_{n-j}}$.
Consider the \index{Euler polynomial} {\it Euler polynomial} of $\alpha$, defined by:
\[ \chi_{\alpha}(t):=\sum_{j \geq 0} \chi(\alpha_j) \cdot (-t)^j.\]
For $c_*:CF(\bC\bP^n)\to A_*(\bC\bP^n)$ the Chern class transformation of MacPherson, let
\[
c_*(\alpha)=\sum_{j\geq 0} c_j [\bC\bP^j] \in A_*(\bC\bP^n),
\]
and define the corresponding \index{Chern polynomial} {\it Chern polynomial} of $\alpha$ by
\[
c_\alpha(t):= \sum_{j\geq 0} c_j t^j.
\]
In \cite{Al},  Aluffi showed that the polynomials $\chi_{\alpha}(t)$ and $c_\alpha(t)$ carry precisely the same information. (A similar result was obtained a decade earlier by Ohmoto in \cite{Oh}, but here we make use of Aluffi's formulation.) More precisely, one has the following.
\bt[\cite{Al}] \label{Aif}
The involution on polynomials (of the same degree)
\begin{center} $p(t) \longmapsto \mathcal{I}(p)(t):=\frac{t \cdot p(-t-1)+p(0)}{t+1},$
\end{center}
interchanges $c_\alpha(t)$ and $\chi_\alpha(t)$, i.e.,
 ${c_\alpha=\mathcal{I}(\chi_\alpha)}$ and 
    ${\chi_\alpha=\mathcal{I}(c_\alpha)}.$
\et

Back to the proof of the Involution Conjecture (Theorem \ref{HSconj}), let $\alpha:=\Eu_{X^\circ}$, regarded as a constructible function on $\bC\bP^n$. 
Our geometric interpretation of $c_*^{Ma}(X^\circ):=c_*(\Eu_{X^\circ})$ from Theorem \ref{mrww1} yields (with $d=\dim X$) the following identities:
\[
B_X(\pp, \uu)=(-1)^{d}c_{\alpha}\left(-\frac{\uu}{\pp}\right)\pp^n,\]
\[
S_X(\pp, \uu)=(-1)^{d}\chi_{\alpha}\left(\frac{\uu}{\pp}\right)\pp^n,
\]
Together with Aluffi's involution formula of Theorem \ref{Aif}, the above identities imply formulae \eqref{st1} and \eqref{st2}, thus proving the Involution Conjecture.


\subsection{Other developments}

\subsubsection{ML degree of mixture of two independence models}

As an application of the topological formula for the ML degree of Theorem \ref{mlds}, the second and third authors found iterated formulae for the ML degree of the variety representing the mixture of two independence models. Consider $\mathbb{C}^{mn}$ as the space of $m$ by $n$ matrices with complex number entries. Let $\mathcal{M}_{mn}\subset \mathbb{C}^{mn}$ be the subvariety corresponding to matrices of rank at most 2 and the sum of all entries being equal to 1. Let $\mathcal{M}_{mn}^\circ=\mathcal{M}_{mn}\cap (\mathbb{C}^*)^{mn}$.  Rodriguez-Wang have proved the following formula conjectured by Hauenstein, Rodriguez and Sturmfels in \cite{HRS}. 
\begin{theorem}{\rm \cite[Theorem 3.12]{RW}} For $n\geq 3$,
\[
\mathrm{MLdeg}(\mathcal{M}_{3n}^\circ)=2^{n+1}-6.
\]
\end{theorem}

\subsubsection{Computing local Euler obstruction from sectional ML degrees}
This survey focusses on topological methods to study the degree of an  optimization problem. 
The other direction where optimization degrees are used to gain  insights on invariants is also of interest. For instance, we now discuss how the maximum likelihood degree is  used to compute the \index{local Euler obstruction} local Euler obstruction function of a variety at a point. 

Let $\T=\CStarN$ be an affine complex torus with coordinates $z_1,\dots,z_N$. 
Let $X$ be a closed pure dimensional (not necessarily irreducible) subvariety of $T$.
 Let $\linearf$ denote a linear function such that its zero set is a hyperplane  $\Hf$ in $\T$.
 Then, we have a natural closed embedding of $T\setminus \Hf$ to $\CStarNplusOne$ given by 
$$(z_1, \ldots, z_N, f): T\setminus \Hf\to \CStarNplusOne.$$ 
For a closed subvariety $X$ of $T$ and a hyperplane $\Hf\subset T$, we define $\MLdeg(X\setminus \Hf)$ to be the maximum likelihood degree of $X\setminus \Hf$ as a closed subvariety of $\CStarNplusOne$ via the above embedding as in Definition~\ref{definition:mldegree-very-affine}.

\begin{theorem}\label{thm:ml-to-local-euler}{\rm\cite[Theorem 1.7]{RW18}}
Let $X$ be a pure $d$-dimensional closed subvariety of $\T$, and let $P\in X$ be any closed point.
Furthermore, let $\Hone, \ldots, \Hd$ denote general hyperplanes in $T$ passing through $P$.
For  $k\in\{0,1,\dots,d+1\}$, define
 the $k$-th \index{removal ML degree} {\it removal ML degree with respect to} $P$ by  
$$r_k(P,X):=\MLdeg\left(X\cap \Hone\cap\Htwo\cap\dots\cap\HkMinusOne\setminus\Hk\right),$$ 
with the conventions   
$r_0(P,X)=\MLdeg(X)$ and $r_1(P,X)=\MLdeg\left(X\setminus \Hone\right)$.
Then, $\Eu_X(P)$ is given by an alternating sum of removal ML degrees
\begin{equation}\label{eq:MLobfun}
\Eu_X(P)  :=
(-1)^d r_0(P,X)+(-1)^{d-1} r_1(P,X)+\cdots+r_d(P,X)-r_{d+1}(P,X).
\end{equation}
for any point $P\in X$. 
\end{theorem}

\begin{remark}
When $f$ is given by the sum of the coordinates, then 
$$r_1(P,X)=\MLdeg\left(X\setminus \Hone\right)$$ is the ML degree appearing in Definition~\ref{definition:mldegree-statistics} 
\end{remark}

\begin{example}

Consider a general very affine curve $X$ of degree $d$ in 
$\left(\mathbb{C}^*\right)^2$. 
We have the following values for the removal ML degrees, where $P_0$ is a general point in $\left(\mathbb{C}^*\right)^2$ and $P_1$ is a smooth point on the curve:
$$
\begin{array}{rccc|ccc}
k: & 0 & 1 & 2 &\Eu_X \\
\hline
r_{k}\left(P_0,X\right): & d^2  & d^2+d &  d&	0\\
r_{k}\left(P_1,X\right): & d^2  & d^2+d &  d-1&	1.\\
\end{array}
$$

If the very affine curve $X$ is the nodal cubic we have the following values, where 
$P_2$ is the singular points of the curve:
$$
\begin{array}{rccc|ccc}
k: & 0 & 1 & 2 &\Eu_X \\
\hline
r_{k}\left(P_0,X\right): &7  & 10 & 3&	0\\
r_{k}\left(P_1,X\right): & 7  & 10 &  2&	1\\
r_{k}\left(P_2,X\right): & 7  & 10 &  1&  2.\\
\end{array}
$$

\end{example}

\subsubsection{ML degree of a sparse polynomial system}
Sparse polynomial systems appear in many areas of applied algebraic geometry. 
In this subsection we mention how the ML degree of a large class of models is determined by a mixed volume of Newton polytopes  of a sparse polynomial system called the Lagrange likelihood equations. 

Let $A_1,\dots, A_k$ denote nonempty finite subsets of $\mathbb{N}_{\geq 0}^k$. 
Suppose $f_1,\dots, f_k$ are polynomials 
of the form
\[
f_i(p) =\sum_{a\in A_i} c_{i,a} p^a, \quad i\in\{1,\dots, k\}
\]
and the coefficients $\{c_{i,a}\}$ are generic. (Note, we use the standard multi-index notation $p^a$ for the monomial with exponent vector $a$.) 
The system of equations $f_1(p)=\dots=f_k(p)=0$ is 
said to be a \index{sparse polynomial system}  {\it sparse polynomial system} in $n$ unknowns $p=(p_1,\dots,p_n)$ with monomial supports $A_1,\dots,A_k$. 
The generic coefficients condition ensures that $V(f_1,\dots,f_k)\cap(\CC^*)^n$ is either empty or reduced with codimension~$k$. 

For  $(u_1, \dots,u_n)\in \CC^{n}$, 
the \index{Lagrangian function} {Lagrangian function} for  maximum likelihood estimation on the algebraic model $V(f_1,\dots,f_k)$ is defined as 
\begin{align}
    \lagObjective(p_1,\ldots, p_n, \lambda_1,\ldots, \lambda_k) &:= \sum_{i=1}^n u_i \log(p_i) - \sum_{j=1}^k \lambda_j f_j(p).
\end{align}
The partial derivatives of $\lagObjective$~are 
\begin{align}
\frac{\partial}{\partial p_i} \lagObjective &= \frac{u_i}{p_i} - \frac{\partial}{\partial p_i} \Big( \sum_{j=1}^k \lambda_j f_j \Big), &&\quad i =1,\dots, n, \\
\frac{\partial }{\partial \lambda_j} \lagObjective &= -f_j, &&\quad j =1,\dots,k. \label{eq:partiallambda}
\end{align}
After clearing denominators of the partial derivatives, we  have a sparse polynomial system in $n+k$ unknowns that we denote by  $\nabla\lagObjective=0$ and call the  \index{ Lagrange likelihood equations} Lagrange likelihood equations.
The solutions to $\nabla\lagObjective=0$ with $p_1\dots p_n \neq 0$
correspond to the set of critical points of $\sum_{i=1}^n u_i \log(p_i)$ restricted to $V(f_1,\dots,f_k)\cap (\bC^{*})^n$. 
Thus, counting the number of solutions to 
$\nabla\lagObjective=0$ with $p_1\dots p_n \neq 0$ is the ML degree of $V(f_1,\dots,f_k)$, the number of critical points of the the log-likelihood function $\sum_{i=1}^n u_i \log(p_i)$ on $V(f_1,\dots,f_k)\cap (\bC^{*})^n$.

Recall that if  $g=\sum_{a\in A} c_a x^a$  is a sparse polynomial, 
then the \index{Newton polytope} {\it Newton polytope} of $g$ is the convex hull of the exponent vectors of $A$.
Given $m$ convex bodies 
$K_1,\dots,K_m$ 
in $\mathbb{R}^m$, 
and positive real numbers $\mu_1,
\dots,\mu_m$,
the volume of the \index{Minkowski sum} Minkowski sum
$\mu_1 K_1 +\cdots +\mu_m K_m$
as a function of $\mu_1,\dots, \mu_m$
is a homogeneous polynomial $Q(\mu_1,\dots,\mu_m)$ of degree $m$.
The {\it mixed volume} of 
$K_1,\dots,K_m$
is defined to be 
$\frac{1}{m!}$ times
the coefficient of $\mu_1\cdots\mu_m$ in $Q$.
For more details about mixed volumes see, e.g., ~\cite{E}. 

With the notation now set,  
we state the main result and corollary of \cite{LNRW2023}.

\begin{theorem}{\rm \cite[Theorem 2.2]{LNRW2023}}\label{theorem:MLsparse-main}
For general sparse polynomials $F= ( f_1,\dots,f_k)$, 
the ML degree of $V(F)$ equals the mixed volume 
of the Newton polytopes of the polynomials in the system $\nabla \lagObjective=0$.
\end{theorem}
 
\br 
The proof of this result in \cite{LNRW2023} 
relies on lemmas such as showing that any solution of  $\nabla \lagObjective =0$ must have nonzero $\lambda$ coordinates. 
The more difficult part is addressing the fact that the sparse system  $\nabla\lagObjective=0$ does not have generic coefficients because of the dependencies of the coefficients of $\frac{\partial}{\partial \lambda_j}\lagObjective$ on $\frac{\partial}{\partial p_i}\lagObjective$.
\er

Moreover, the generic coefficients hypothesis on $f_1$ can be relaxed when  $f_1(p)= p_1 + \ldots + p_n -1$  \cite[Remark 2.23]{LNRW2023}. This sum to one constraint appears  
in Definition~\ref{definition:mldegree-statistics} with $p=(p_0,\dots,p_n).$

\bc{\rm \cite[Corollary 2.14]{LNRW2023}}
Consider two general sparse polynomial systems: 
$F =( f_1,\ldots, f_k )$ and 
$G = ( g_1,\ldots, g_k ) $, 
where the monomial supports are $A_1,\dots, A_k$ and $B_1,\dots, B_k$ 
respectively.  
If $\conv(A_i)=\conv(B_i)$ for $i=1,\dots, k$,
then the ML degree of $V(F)$ equals the ML degree of $V(G)$.
\ec
This corollary is surprising because the Newton polytope of the Lagrange likelihood equations can be vastly different even though $\conv(A_i)=\conv(B_i)$ for $i=1,\dots,k$ as shown in \cite[Example~2.15]{LNRW2023}.

\subsubsection{ML data discriminant for positive real solutions}

Throughout this survey we have focussed on the algebraic degree of an optimization problem.
The ML degree is an intrinsic measure of complexity for solving the likelihood equations. A statistician cares about finding real solutions in the probability simplex. The space of data can be partitioned into full dimensional open cells by taking the complement of a hypersurface known as the \index{data discriminant} ML data discriminant. 
For each open cell, the number of complex critical points with real coordinates strictly greater than zero is constant. For details see \cite{RT-discriminant, RT-discriminant-2}.

\subsubsection{Gaussian models and symmetric matrices}

Likelihood geometry extends beyond the discrete models that we have discussed thus far. 
In this section the ambient space for our statistical models is the cone of $m\times m$ positive definite matrices $\PD$ (instead of the probability simplex). 
The motivation comes from the fact that Gaussians with their mean $\mu$ centered at zero are determined by their covariance matrix $\Sigma$ (which is positive definite). For concreteness, recall that the Gaussian density is given by
\begin{equation*}
    f_\Sigma( \un x)= \frac{1}{\sqrt{\det(2\pi \Sigma)}} \exp \left( -\frac12 {\un x}^\top \Sigma^{-1} \un x \right), \quad \un x \in \RR^m.
\end{equation*}
For these models, the maximum likelihood estimation problem takes the following form. 
We assume  
i.i.d. data $X_1,...,X_N$, where each random variable $X_i$ has a Gaussian distribution 
$\mathcal{N}(0,\Sigma)$ of mean $0$, covariance $\Sigma$, 
and $N$ is the sample size. 
Our data is a collection of $N$ random vectors in $\RR^m$, and the i.i.d. assumptions allow us work with
the \index{sample covariance matrix} {\it sample covariance matrix} 
\begin{equation}
    S = \frac{1}{N} X_i X_i^\top,
\end{equation}
whereas in the discrete setting our data was a vector of counts. 
The \index{log-likelihood} {\it log-likelihood function} for \index{Gaussian model} Gaussian models takes the form (see, e.g., \cite[Section 7.1]{Su})
\begin{equation}\label{eq:Gaussianlik}
    \ell(\Sigma) = -\frac{n}{2}(\log\det(\Sigma)+ \mathrm{tr}(S\Sigma^{-1})).
\end{equation}
The gradient of $\ell$ is $\nabla(\ell)(\Sigma)$ is proportional to  $\Sigma - S$, so that $\hat{\Sigma} = S$ is the global optimum whenever $S \in \mathcal{M}$. 

This setup leads to a wide range of models whose likelihood geometry can be studied.
This include directed graphical models \cite[Chapter 13]{Su},  undirected graphical models~\cite{SU2010,Uhler}, and more recently linear subspaces of symmetric matrices~\cite{BCEMR,DKRS,EFSS,JKW, AGKM, AZ}.
In addition,  \cite{DLL} has generalized these concepts to 
maximum likelihood degree of a homogeneous polynomial  on a (smooth) projective variety $X$. 


\section{Linear optimization {on a variety}}\label{lin}
{This section is devoted to \index{linear optimization on varieties} optimizing a linear objective function on a variety,}
as well as to the study of the corresponding algebraic degree, which in \cite{MRW6} is called the {\it linear optimization (LO) degree}.

\subsection{Linear optimization degree}

\bd[Linear optimization (LO) degree] \index{linear optimization degree}
The {\it linear optimization (LO) degree} $\LOdeg(X)$ of an affine variety $X\subset \bC^n$ is the number of critical points of a general linear function restricted to the smooth locus $X_{\textrm{reg}}$ of $X$. 
\ed

The LO degree, which is already computed by Theorem \ref{edf}, gives an algebraic measure to the complexity of optimizing a linear function over algebraic models $X_{\reg} \cap \mathbb{R}^n$,
which are prevalent in algebraic statistics and applied algebraic geometry. 

An equivalent definition of the linear optimization degree $\LOdeg(X)$ of an affine variety $X\subset \bC^n$ can be given as follows.   Let $T^*_X\bC^n$ be  the {affine} \index{conormal variety} conormal variety of $X$, i.e., the closure of the conormal bundle $T^*_{X_{\textrm{reg}}}\bC^n$ of $X_{\textrm{reg}}$ in $T^*\bC^n$. Consider the trivialization $T^*\bC^n\cong \bC^n\times \bC^n$ of the cotangent bundle, where the first factor is the base and the second is the fiber. Then the projection of $T^*_X\bC^n$ to the second factor $\bC^n$ is generically a finite map, and its degree is equal to $\LOdeg(X)$.


\subsection{Linear optimization bidegrees and Chern-Mather classes}
Similar to the ML degrees, we can also define LO bidegrees $b_i(X)$ and sectional LO degrees $s_i(X)$, and investigate how these are related.

\bd[LO bidegrees] The \index{linear optimization bidegree} {\it LO bidegrees} of an irreducible affine variety $X\subset \bC^n$, denoted by $b_i(X)$ or simply $b_i$, are defined as the bidegrees of $T^*_X\bC^n$. Specifically, if $\bC^n\times \bC^n\subset \bC\bP^n\times \bC\bP^n$ is the standard compactification, the LO bidegrees of $X$ are the coefficients of the Chow class of the closure $\overline{T^*_X\bC^n}$ of $T^*_X\bC^n$ in $\bC\bP^n\times \bC\bP^n$, that is,
\be\label{bi}
[\overline{T^*_X\bC^n}]=b_0 [\bC\bP^0\times \bC\bP^n]+b_1[\bC\bP^{1}\times \bC\bP^{n-1}]+\cdots +b_d[\bC\bP^{d}\times \bC\bP^{n-d}]\in A_*(\bC\bP^n\times \bC\bP^n)
\ee
where $d=\dim X$. 
\ed 
In particular, $b_0(X)=\LOdeg(X)$.

Fixing the standard compactification $\bC^n\subset \bC\bP^n$, we regard the local Euler obstruction function $\Eu_X$ of the affine variety $X \subset \bC^n$ as a constructible function on $\bC\bP^n$, with value $0$ outside of $X$. Applying to it  the Chern-MacPherson transformation $c_*:CF(\bC\bP^n) \to A_*(\bC\bP^n)$,  we get as before \index{Chern-Mather class} Chern-Mather class of $X$:
\be\label{cma}
c_*^{Ma}(X):=c_*(\Eu_X)=a_0[\bC\bP^0]+a_1 [\bC\bP^1]+\cdots+a_d [\bC\bP^d]\in A_*(\bC\bP^n).
\ee

For notational convenience, in \eqref{bi} and \eqref{cma} we set $a_j=b_j=0$ if $j\notin \{0, 1, \ldots, d\}$.

In \cite{MRW6}, we describe the relation between the LO bidegrees and the total \index{Chern-Mather class} Chern-Mather class of $X$ as follows.
\bt{\rm \cite[Theorem 1.1]{MRW6}}\label{thm_main}
For any $d$-dimensional irreducible affine variety $X\subset \bC^n$, the sequences $\{a_i\}$ and $\{b_i\}$ defined as in \eqref{bi} and \eqref{cma} satisfy the identity
\be\label{id1}
\sum_{0\leq i\leq d}b_i t^{n-i}= \sum_{0\leq i\leq d}a_i (-1)^{d-i} t^{n-i}(1+t)^{i}.
\ee
\et

The formula in Theorem \ref{thm_main} shows that the Chern-Mather class of the affine variety $X$ is determined by the LO bidegrees. 
The proof of this result uses the same ideas as in the proof of Theorems \ref{thm_CSM} and \ref{mrww1}, based on formula \eqref{eq_main}. However, the relationship is more involved than the corresponding result for the Chern-Mather class of very affine varieties (cf. Theorems \ref{thm_CSM}) since, while the logarithmic cotangent bundle of the pair $(\bC\bP^n, \bC\bP^n\setminus (\bC^*)^n)$ is trivial, the one of $(\bC\bP^n, \bC\bP^n\setminus \bC^n)$ is not. So the logarithmic cotangent bundle $E\coloneqq \Omega^1_{\bC\bP^n}(\log H_\infty)$ is not trivial. Nevertheless, as shown in \cite[Proposition 4.1]{MRW6}, the twisted bundle $E(H_\infty):=E \otimes \cO_{\bC\bP^n}(H_\infty)=\Omega^1_{\bC\bP^n}(\log H_\infty)(H_{\infty})$ is trivial, and formula \eqref{id1} is the result of tracking the relationship between the Chern classes of the closure of $T^*_X\bC^n$ in $E$ and $E(H_\infty)$, respectively.

\br
Let us briefly compare our approach to computing Chern-Mather classes of affine varieties with some of the more classical works \cite{Sab, Al3,PP}.
As above, let $X \subset \bC^n$ be an irreducible affine variety with conormal space $T^*_X\bC^n \subset T^*\bC^n$. Instead of taking the fiberwise projectivization $C(X,\bC^n):=\bP(T^*_X\bC^n) \subset \bP(T^*\bC^n)$ as in, e.g., Sabbah \cite{Sab}, we first compactify the fibers of $T^* \bC^n$ by taking their projective closures, 
i.e., $T^*\bC^n=\bC^n\times \bC^n\subset \bC^n\times \bC\bP^n$, 
so that we keep track of conic subvarieties contained in the zero section of $T^*\bC^n$, and then we compactify $\bC^n\times \bC\bP^n$ using the trivial projective bundle $\bC^n\times \bC\bP^n\subset \bC\bP^n\times \bC\bP^n$. Other authors, like Aluffi \cite{Al3} or Parusi\'nski-Pragacz \cite{PP}, consider the projective closure $\overline{X} \subset \bC\bP^n$ of $X$, together with its corresponding projective conormal variety $C(\overline{X},\bC\bP^n):=\bP(T^*_{\overline{X}}\bC\bP^n) \subset \bP(T^*\bC\bP^n)$. 

As already indicated in Theorem \ref{Sabbah}, Sabbah's formula \cite[Lemme 1.2.1]{Sab} applied to $X \subset \bC^n$ computes the Chern-Mater class of $X$ in the Borel-Moore homology (or Chow group) of $X$. The same formula applied to $\overline{X} \subset \bC\bP^n$ computes the Chern-Mather class of $\overline{X}$ in the Borel-Moore homology (or Chow group) of $\overline{X}$, and resp., of  $\bC\bP^n$, upon using the proper pushforward. By contrast, 
we relate our compactification of $T^*\bC^n$ in $\bC\bP^n\times \bC\bP^n$ to the twisted logarithmic cotangent bundle $\Omega^1_{\bC\bP^n}(\log H_\infty)(H_{\infty})$ of $(\bC\bP^n, H_\infty)$, and compute the Chern-Mather class of $X$ in $A_*(\bC\bP^n)$ via Theorem~\ref{thm_m} and Ginsburg's microlocal interpretation of Chern classes.
\er

\smallskip


The equality of top degree coefficients in \eqref{id1} reproves the following result of Seade-Tib\u{a}r-Verjovsky \cite[Equation (2)]{STV}, already recalled in Theorem \ref{edf}.
Hence Theorem~\ref{thm_main} can be viewed as a 
higher dimensional generalization of this~result.
\bc\label{cor_linear}
If $X\subset \bC^n$ is a $d$-dimensional irreducible affine variety and $H\subset \bC^n$ is a general affine hyperplane, one has
\be\label{eq_b0}
b_0=(-1)^{d} \cdot \chi(\Eu_{X}|_{\bC^n\setminus H}).
\ee
\ec

Moreover, by plugging $t=-1$ in \eqref{id1}, one gets the following relation between the value of the local Euler obstruction function of an affine cone at the cone point, and the LO bidegrees of the affine cone. 
This formula has already appeared 
in work of L\^e-Teissier \cite{TT}. 

\bc\label{cor_cone}
Let $X$ be an \index{affine cone} affine cone of a projective variety, and denote its cone point by $O$. Then
\be\label{euc}
\Eu_X(O)=b_d(X)-b_{d-1}(X)+\cdots +(-1)^db_0(X),
\ee
with $d=\dim X$.
\ec

\subsection{Sectional linear optimization degrees. Relation to LO bidegrees}
By analogy with the sectional maximum likelihood degrees, sectional linear optimization (LO) degrees of an affine variety were introduced in \cite{MRW6}. We recall their definition and explain how they relate to the LO bidegrees.

\bd[Sectional linear optimization (LO) degrees] Let $X\subset \bC^n$ be a $d$-dimensional irreducible affine variety. For any $0\leq i\leq d$, \index{sectional linear optimization degree} the {\it $i$-th sectional  LO degree} of $X$, denoted by $s_i(X)$ or simply $s_i$, is given by
\be\label{si} s_i(X):= \LOdeg(X\cap H_1\cap \cdots \cap H_i),\ee
where $H_1, \ldots, H_i$ are generic affine hyperplanes. 
\ed
Note that $s_0(X)=\LOdeg(X)$, and $s_d(X)$ is the degree of $X$. Here, for notational convenience, we also set $s_i=0$ for $i>d$. 

Regarding the relation between the LO bidegrees and sectional LO degrees, we show the following result in \cite{MRW6}.
\bt{\rm \cite[Theorem 1.4]{MRW6}}\label{thm_bs}
Let $X\subset \bC^n$ be any irreducible affine variety, and let $b_i$ and $s_i$ be its LO bidegrees and LO sectional degrees, respectively. Then $s_i=b_i$ for all $i$. 
\et

In particular, the above result gives, via~\eqref{eq_b0} and \eqref{si}, a topological interpretation of all LO bidegrees as
Euler characteristics, that is,
\[
b_i(X)=(-1)^{d-i}\chi(Eu_{X\cap H_1\cap \cdots \cap H_i} |_{\bC^n\setminus H_{i+1}}).
\]
Moreover, formula \eqref{euc} can be reformulated as an alternating sum of sectional LO degrees.

\subsection{Relation to polar degrees}
We discuss here the relation between the LO bidegrees of an affine variety and the polar degrees of its projective closure (see \cite[Section 6]{MRW6} for complete details).

Let $X\subset \bC^n$ be a $d$-dimensional irreducible affine variety, and let $\overline{X} \subset \bC\bP^n$ be its projective closure. Recall that the  \index{conormal variety} conormal variety $T^*_X\bC^n$ is the closure in $T^*\bC^n$ of 
\[
T^*_{X_{\reg}} \bC^n=\left\{(\un x, \un u)\in T^*\bC^n=\bC^n_{\un x}\times \bC^n_{\un u}\mid \un x\in X_{\reg} \text{ and } {\un u}|_{T_{\un x} X_{\reg}}=0\right\}.
\]
Here, we view $\un u=(u_1, \ldots, u_n) \in \bC^n_{\un u}$ as the parallel 1-form $\sum_{1\leq i\leq n}u_i dx_i$ on $\bC^n$. So, if $x\in X_{\reg}$, then ${\un u}|_{T_{\un x} X_{\reg}}=0$ means that $\un x$ is a critical point on $X_{\reg}$ of the linear function $\sum_{1\leq i\leq n}u_i x_i$, or, equivalently, a level set of $\sum_{1\leq i\leq n}u_i x_i$ is tangent to $X_{\reg}$ at $\un x$. 
Let $\bP(T^*_X \bC^n)\subset \bC^n_{\un x} \times \bC\bP^{n-1}$ be the fiberwise projectivitation of $T^*_X \bC^n$, with closure  $\overline{\bP(T^*_X\bC^n)} \subset \bC\bP^n\times \bC\bP^{n-1}$.

On the other hand, cf. \cite{RS2013} (see also \cite{Al3}), the  \index{projective conormal variety} {\it projective conormal variety} $\bP(C_{\overline{X}}\bC\bP^n)$ can be identified with the $(n-1)$-dimensional subvariety $N_{\overline{X}}$ of $\bC\bP^n\times (\bC\bP^n)^\vee$ defined by the closure of 
\[
N_{\overline{X}_{\reg}}=\left\{(\un p, H)\in \bC\bP^n\times (\bC\bP^n)^\vee\mid \un p\in \overline{X}_{\reg} \text{ and } H \text{ is tangent to $\overline{X}_{\reg}$ at }\un p\right\},
\]
where the dual projective space $(\bC\bP^n)^\vee$ parametrizes hyperplanes in $\bC\bP^n$. 

Let $H_\infty\in (\bC\bP^n)^\vee$ denote the hyperplane at infinity in $\bC\bP^n$, and let $\pi_\infty: (\bC\bP^n)^\vee \dashrightarrow \bC\bP^{n-1}$ be the rational map given by  projecting from $H_\infty$. We then have the following.

\bl{\rm \cite[Proposition 6.1]{MRW6}}
Assume that $X$ is not contained in any proper affine subspace, that is, $\overline{X}$ is not contained in a hyperplane. Under the above notations, the rational map $$\id\times \pi_\infty: \bC\bP^n\times (\bC\bP^{n})^\vee\dashrightarrow \bC\bP^n\times \bC\bP^{n-1}$$ restricts to a birational map between $N_{\overline{X}}$ and $\overline{\bP(T^*_X\bC^n)}$. 
\el

The above Lemma allows us to relate the LO bidegrees of $X$ and the polar degrees of $\overline{X}$. Recall that the LO bidegrees $b_i(X)$ (or simply $b_i$) are the bidegrees of the closure of the {affine} conormal variety $T^*_X\bC^n$ in $\bC\bP^n\times \bC\bP^n$, i.e., they are defined by the following formula
\[
[\overline{T^*_X\bC^n}]=b_0 [\bC\bP^0\times \bC\bP^n]+b_1[\bC\bP^{1}\times \bC\bP^{n-1}]+\cdots +b_d[\bC\bP^{d}\times \bC\bP^{n-d}]\in A_*(\bC\bP^n\times \bC\bP^n),
\]
where $d=\dim X$.
Similarly, the \index{polar degree} {\it polar degrees} $\delta_i(\overline{X})$ (or simply $\delta_i$) of $\overline{X}$ are the bidegrees of the projective conormal variety $N_{\overline{X}}\subset \bC\bP^n\times (\bC\bP^n)^\vee$. More precisely, they are defined by (see, e.g., \cite[Section 2]{Sturmfels})
\[
[N_{\overline{X}}]=\delta_1 [\bC\bP^{0}\times \bC\bP^{n-1}]+\cdots +\delta_{d+1}[\bC\bP^{d}\times \bC\bP^{n-d-1}].
\]
Polar degrees have been used in \cite[Theorem 5.4, Theorem 6.11]{DHOST} to bound (or to compute under certain transversality assumptions) both the ED degree of $X$ and the projective ED degree of $\overline X$.

The following result was proved in \cite{MRW6}.
\bp{\rm \cite[Proposition 6.2]{MRW6}}\label{prop_iff}
The bidegrees of $T^*_X\bC^n\subset \bC^n_{\un x}\times \bC^n_{\un u}$ and the bidegrees of $N_{\overline{X}}\subset \bC\bP^n\times (\bC\bP^n)^\vee$ coincide in the sense that
\be\label{eq_delta}
b_i(X)=\delta_{i+1}(\overline{X}), \quad\text{for } \ 0\leq i\leq d,
\ee
if and only if the hyperplane at infinity $H_\infty$ is not a point in the dual variety $\overline{X}^\vee\subset (\bC\bP^n)^\vee$.
\ep

Combining Theorem \ref{thm_bs} and Proposition \ref{prop_iff}, one gets the following generalization of \cite[Theorem 13]{wasserstein} (see also \cite[Proposition 2.9]{Sturmfels}) to singular varieties. 
\bc\label{cor_si} Let $X\subset \bC^n$ be an affine variety, with projective closure 
$\overline{X} \subset \bC\bP^n$. 
Assume that the hyperplane at infinity $H_\infty$ is not contained in $\overline{X}^\vee$. Then the sectional LO degrees of $X$ coincide with the polar degrees of $\overline{X}$, that is, $s_i(X)=\delta_{i+1}(\overline{X})$ for all $0\leq i\leq \dim X$. 
\ec
\br
If the affine variety $X\subset \bC^n$ is defined by homogeneous polynomials, i.e., $X$ is the cone of a projective variety, then its closure intersects the hyperplane at infinity $H_\infty$ transversally. In this case, $H_\infty$ is not contained in $\overline{X}^\vee$, hence $\delta_{i+1}(\overline{X})=s_i(X)=b_i(X)$, for all  $0\leq i\leq \dim X$.
For example, if $X\subset \bC^9$ 
is defined by the vanishing of the determinant of the matrix 
$\begin{bsmallmatrix} 
x_0&x_1&x_2\\
x_3&x_4&x_5\\
x_6&x_7&x_8\\
\end{bsmallmatrix}$, 
then the 
LO bidegrees of $X$ 
and the polar degrees of $\overline{X}$ are given by
\begin{align*}
[\overline{T^*_X\bC^9}]&=
6 [\bP^0\times \bP^9]+   12 [\bP^1\times \bP^8]+   12 [\bP^2\times \bP^7]+   6 [\bP^3\times \bP^6]+   3 [\bP^4\times \bP^5],\\
[N_{\overline{X}}] & =
6 [\bP^0\times \bP^8]+   12 [\bP^1\times \bP^7]+   12 [\bP^2\times \bP^6]+   6 [\bP^3\times \bP^5]+   3 [\bP^4\times \bP^4].
\end{align*}
\er

The following example shows that when $H_\infty\in \overline{X}^\vee$, the two sets of bidegrees considered above are different. 

\bex
Let $X$ in $\bC^3$ be the smooth curve $V(x^2+y^2+z^2-1,y-x^2)$. 
Its projective closure $\overline X=V(x^2+y^2+z^2-w^2,yw-x^2)$ is smooth, while 
the dual variety $\overline X^\vee$ is a singular hypersurface defined by an octic  polynomial with $49$ terms.
The first terms of this octic in the dual coordinates are
\[
16 (\dot y^{2}+\dot z^{2})\dot w^6
- 8 \dot y(\dot x^{2}+4\,\dot y^{2}+4\,\dot z^{2})
\dot w^5+\dots .
\]
Since the octic vanishes at the point $[\dot x: \dot y: \dot z: \dot w] =[ 0:0:0:1]$,
$H_\infty$ is in the dual variety $\overline X^\vee$.
Hence,  the LO degrees and $X$ and polar degrees of $\overline X$ do not coincide. 
Indeed, as computed with \index{Macaulay2} \texttt{Macaulay2} \cite{GS-M2}, the LO bidegrees of $X$ 
and the polar degrees of $\overline{X}$ are given by
\begin{align*}
[\overline{T^*_X\bC^3}]&=6 [\bP^0\times \bP^3]+4[\bP^{1}\times \bP^{2}],\\
[N_{\overline{X}}] & =8 [\bP^{0}\times \bP^{2}] +4[\bP^{1}\times \bP^{1}].
\end{align*}
\eex

\section{Non-generic Data. Morsification and applications}\label{nong}
Typically, results on Euclidean distance degrees and nearest point problems have a hypothesis requiring genericity of the data point $\un u$, or one studies ED-discriminant loci (roughly speaking, the collection of data points $\un u$ for which the function $d_{\un u}$ has a different number of critical points than the ED degree), e.g., see \cite{DHOST}.  There are many practical situations when data is not generic, e.g., when the data is sparse. 

Working with non-generic data for $X \subset \bC^n$ can lead to the distance function having an infinite number of critical points or even a positive dimensional critical set. For instance, every point on the circle is a critical point of the distance function when the data is taken to be the center of the circle. 


\subsection{Morsification}
Results of \cite{MRW4} allow one to handle situations when the data belongs to the ED-discriminant (i.e., data is not generic) by observing the ``limiting'' behavior of critical sets obtained for generic choices of data. 
Specifically, by adding some {\it noise} ${\un \epsilon} \in \C^n$ to an arbitrary data point $\un u$, one is back in the generic situation, and the limiting behavior of critical points of $d_{{\un u}+t {\un \epsilon}}$ on $X_{\rm reg}$ for $t \in \bC^*$ (with $\vert t \vert$ very small), as $t$ approaches the origin of $\bC$, yields valuable information about the initial nearest point problem. 
Notice that one can write in this case $$d_{{\un u}+t {\un \epsilon}}(\un x)=d_{\un u}(\un x)- t\ell(\un x) + c,$$
with $\ell(\un x)=2 \sum_{i=1}^n  \epsilon_i  x_i$ and $c$ is a constant with respect to $\un x$.
So the critical points of $d_{{\un u}+t {\un \epsilon}}$ coincide with those of $d_{\un u}-t\ell$. 
Moreover, since  ${\un \epsilon}$ is generic, $\ell$ is a generic linear function.

The observation of the previous paragraph places us at the origins of the following Morsification procedure. \index{morsification}
Let $f\colon \C^n \to \C$ be a polynomial function, and let $\ell\colon \bC^n \to \bC$ be a linear function. Let $X \subset \C^n$ be a possibly singular closed irreducible subvariety such that $f$ is not constant on $X$, and restrict $f$ and $\ell$ to $X$. If the linear function $\ell$ is general enough, then the function $f_t:=f-t\ell$ is a holomorphic Morse function on $X_{\rm reg}$ (i.e., it has only non-degenerate isolated critical points) for all but finitely many $t \in \C$. Motivated by the above NPP, one is then  interested in studying the limiting behavior of the set of critical points of $f_t\vert_{X_{\reg}}$ as $t$ approaches $0 \in \bC$.

If $X=\C^n$ and $f\colon \C^n \to \C$ is a polynomial function with only isolated singularities $P_1,\ldots,P_r$, a solution to the above problem is provided by the classical Morsification picture, as shown by Brieskorn in  \cite[Appendix]{Br}. More precisely, if $\mu_i$ is the {Milnor number} of $f$ at $P_i$ (cf. \cite{Mil68}), then in a small neighborhood of $P_i$ the function $f_t$ has $\mu_i$ Morse critical points which, as $t$ approaches $0$, collide together at $P_i$.

In the general situation, for $X$ and $f$ with arbitrary singularities, this limiting behavior of the set of critical points of $f_t\vert_{X_{\reg}}$ can be studied by using constructible functions and vanishing cycle techniques, together with work of Ginsburg \cite{Gin} on characteristic cycles. 
Let $\Eu_X$ be the local Euler obstruction function on $X$, regarded as a constructible function on $\bC^n$ by extension by zero. Then $\Eu_X$ is constructible on $\bC^n$ for any Whitney stratification of $X$, to which one adds the open stratum $\bC^n \setminus X$.  We endow $X$ with a Whitney stratification $\mathscr X$ with finitely many strata, with respect to which the stratified singular set of $f$ is defined as $\Sing_{\mathscr X}f:=\bigcup_{V\in \mathscr X} \Sing (f\vert_{V})$. The set $\Sing_{\mathscr X}f$ is then a closed set in $X$ distributed in a finite number of critical fibers of $f$. If $c \in \bC$ is a critical value of $f\colon X \to \bC$ and $\varphi_{f-c}$ denotes the corresponding vanishing cycle functor of constructible functions, then $\varphi_{f-c}(\Eu_X)$ is a constructible function supported on $\Sigma_c:=\{f=c\} \cap \Sing_\mathscr{X} f$. Each $\Sigma_c$ gets an induced Whitney stratification witnessing the constructibility of $\varphi_{f-c}(\Eu_X)$. One can thus refine $\mathscr X$ to a Whitney stratification $\mathscr S$ of $X$ adapted to the functions $\varphi_{f-c}(\Eu_X)$, for all critical values $c$ of $f$, and we may use the notation  $\Sing_{\mathscr S}f$ instead of $\Sing_{\mathscr X}f$ whenever strata in each $\Sigma_c$ need to be taken into account. (Note that $\Sing_{\mathscr S}f =\Sing_{\mathscr X}f$ as sets.) Using the distinguished basis of constructible functions consisting of local Euler obstructions of closures of strata, one can then introduce uniquely determined integers $n_V$ for each stratum $V \subset \Sing_\mathscr{S} f$ so that:
\be\label{fcfi}
\sum_{c \in \bC} \varphi_{f-c}(\Eu_X)=\sum_{V \subset \Sing_\mathscr{S} f} (-1)^{\codim V-1}\cdot n_V \cdot \Eu_{\overline{V}}.
\ee
It can be shown that $n_V \geq 0$, for any $V \subset \Sing_\mathscr{S} f$ (see \cite{MRW4} and the references therein).

We next recall the definition of the limit of a family of sets of points (cf. \cite{MRW4}).
\bd \index{limit of sets}
(i) A {\it set of points} 
is a finite set endowed with a multiplicity function, i.e., after fixing a ground set $S$, a set of points $\cM$ of $S$ is given by a function $\cM: S\to \mathbb{Z}_{\geq 0}$ such that $\cM(x)=0$ for all but finitely many $x\in S$. The value $\cM(x)$ is called the {\it multiplicity} of $\cM$ at $x$. \index{multiplicity} 

(ii) Let $\cM_t$ be a family of sets of points on a ground set $S$, parametrized by $t\in D^*$, with $D^*$ a punctured disc centered at the origin. The {\it limit} $\lim_{t\to 0}\cM_t$ of $\cM_t$ as $t\to 0$ is defined as the set of points given by:
\[
(  \lim_{t\to 0}\cM_t  )  (x)\coloneqq \varprojlim_{U}  \lim_{t\to 0}\sum_{y\in U}\cM_t(y),
\]
where $\varprojlim_U$ denotes the inverse limit over all open neighborhood of $x$. 
\ed

The main result of \cite{MRW4} can now be stated as follows.
\bt{\rm \cite[Theorem 1.3]{MRW4}} \label{thm_noniso}\ 
In the above notations, we have
\begin{equation}\label{eq_maini}
\lim_{t\to 0}\Sing(f_t|_{X_{\reg}})=\sum_{V \subset \Sing_\mathscr{S} f}n_{V}\cdot\Sing(\ell|_{V}).
\end{equation}
\et

The left hand side of \eqref{eq_maini} does not take into account the points of $\Sing(f_t|_{X_{\reg}})$ which ``escape at infinity'' as $t \to 0$, i.e., the singular points of $f$ on $X_{\reg}$ which are outside a sufficiently large ball centered at the origin for sufficiently small $t$.

An easy consequence of Theorem \ref{thm_noniso} is the following. 
\begin{corollary}
Let $X\subset \bC^n$ be an irreducible affine variety and let $f\colon \C^n\to \C$ be a polynomial function. Assume that $Z$ is an irreducible component of $\Sing(f|_{X_{\reg}})$ and denote its closure in $\C^n$ by $\overline{Z}$. If $\LOdeg(\overline{Z})>0$, then for a general linear function $\ell\colon \C^n\to \C$, there exists a point $P$ in $\lim_{t\to 0}\Sing(f_t|_{X_{\reg}})$ which is contained in $Z$, but not contained in any other irreducible component of $\Sing(f|_{X_{\reg}})$. 
\end{corollary}
\begin{proof}
Under the notations of Theorem \ref{thm_noniso}, as $Z$ is an irreducible component of $\Sing(f|_{X_{\reg}})$, there must be a stratum $V \subset \Sing_\mathscr{S} f$ such that $V$ contains a nonempty Zariski open subset of $Z$. It suffices to show that the corresponding $n_V$ is positive. Using a transversal slice, this statement can be reduced to the case when $Z=V$ is an isolated point (e.g., see \cite[Proposition 10.4.9 (3)]{MS}). In this case, it follows from the fact that the Milnor number of an isolated hypersurface singularity is always positive (e.g., see \cite[Example 10.4.17 and Example 10.3.58]{MS}). 
\end{proof}

\smallskip

Let us now get back to the calculation of the ED degree of the affine variety $X \subset \bC^n$. In this case, $f=d_u$ is the squared Euclidean distance function, but we allow $u \in \bC^n$ to be {\it arbitrary} (e.g., contained in the ED-discriminant). For $\ell$ a general linear function, using the graph embedding and Theorem \ref{edf}, we first get that 
\be {\rm EDdeg}(X) = \# \Sing(f_t|_{X_{\reg}}),\ee
with $\#$ denoting the cardinality of a set.
Hence, if no points of $\Sing(f_t|_{X_{\reg}})$ go to infinity as $t \to 0\in \bC$, a formula for ${\rm EDdeg}(X)$ can be deduced from Theorem \ref{thm_noniso} in  terms of the multiplicities $n_V$ as (cf. \cite[Corollary 1.9]{MRW4}):
\be\label{edi} \index{Euclidean distance degree}
{\rm EDdeg}(X)=\sum_{V \subset \Sing_\mathscr{S} f}n_{V}\cdot \#\Sing(\ell|_{V}).
\ee

\br
Let us note that if $f$ has only isolated stratified singularities on $X$, this approach is closely related to {\it homotopy continuation}, e.g., see \cite{AG}. However, we also consider here the more general situation when $f$ is allowed to have a positive-dimensional stratified singular locus, in which case 
the limit as $t\to 0\in \bC$ only picks up a finite set of points in the critical locus of $f$, these are among the stratified critical points of a general linear function $\ell$. 
\er

\subsection{Computing multiplicities} 

Formulas \eqref{eq_maini} and  \eqref{edi} emphasize the need for computability of the multiplicities $n_V$, which measure the asymptotics of singularities in a Morse perturbation of $f$. 

Consider the following simple case.
\bex\label{str_iso} 
	Let $X\subset \bC^n$ be an arbitrary complex affine variety, and assume that the polynomial function $f\colon X\to \bC$ is nonconstant and has only isolated stratified critical points $P_1, \ldots, P_r$. Then formula \eqref{fcfi} becomes:
	\be\label{unu}\sum_{c \in \bC} \varphi_{f-c}(\Eu_X)=(-1)^{\dim X-1} \sum_{i=1}^r n_{P_i}\cdot \Eu_{P_i},\ee
	with $n_{P_i}$ given by
	\be\label{niis} n_{P_i}=(-1)^{\dim X-1}\varphi_{f-f(P_i)}(\Eu_X)(P_i) 
	= :(-1)^{\dim X} \Eu_{f-f(P_i)} (X,P_i).
	\ee
The last term in \eqref{niis} is the \index{relative Euler obstruction} {\it relative Euler obstruction} of the function $f-f(P_i)$ on $X$ at $P_i$ (as introduced in \cite{BMPS}).
	Theorem \ref{thm_main} specializes in this case to  
	\be\label{sta}
\lim_{t\to 0}\Sing(f_t|_{X_{\reg}})=\sum_{i=1}^r n_{P_i} \cdot P_i ,
\ee
with $n_{P_i}$ computed as in formula \eqref{niis}.
If $X$ is smooth, then $\Eu_X=1_X$, and formula \eqref{niis} yields that $n_{P_i}$ equals the Milnor number $\mu_{P_i}$ of $f$ at $P_i$, as predicted by the classical Morsification picture \cite{Br}. 
\eex

An explicit calculation of the local multiplicities $n_V$ of formula \eqref{eq_maini} in terms of the geometry and topology of the pair $(X,f)$ is difficult in general.  In \cite{MT}, the first author and Tib\u{a}r introduced the integers
$$\mu_V=\varphi_{f-c}(\Eu_X)(V),$$
i.e., the values of the constructible function $\varphi_{f-c}(\Eu_X)$ along critical strata $V \subset \Sing_\mathscr{S} f$ of $f$, to produce the following  formula for the multiplicities $n_V$ (generalizing Example \ref{str_iso}): 
\bt{\rm \cite[Theorem 1.1]{MT}}\label{nvgeni}
Let $X \subset \bC^n$ be a  complex affine variety, and $f\colon X \to \bC$ the restriction to $X$ of a polynomial function. Then, for any critical value $c$ of $f$, the \index{multiplicity}  multiplicities $n_V$ for singular strata $V \subset f^{-1}(c)$ are given by:
\begin{equation}\label{nvgi}
n_V=(-1)^{\codim V -1} \{ \mu_V - \sum_{ \{S \mid V \subset \overline{S} \setminus S \} } \chi_c(\lk_{\overline S}(V)) \cdot \mu_S \},
\end{equation}
where:
\begin{enumerate}
\item[(i)] the summation is over singular strata $S$ in $f^{-1}(c)$, different from $V$, which contain $V$ in their closure.
\item[(ii)] $\chi_c(\lk_{\overline S}(V))$ is the compactly supported Euler characteristic of the complex link of $V$ in $\bar S$, for a pair of singular strata $(V,S)$ in $f^{-1}(c)$, with $V \subset \overline{S}\setminus S$.
\end{enumerate} 
\et

Formula \eqref{nvgi} is a direct application of Ginsburg's formula for the characteristic cycle of a bounded constructible complex (or constructible function), see, e.g., \cite[Sect.8.2]{Gin}.
It becomes quite explicit if  $X$ is smooth, since in this case 
$\mu_V=\chi(\widetilde{H}^*(F_V;\bC))$ 
is just  the Euler characteristic of the reduced cohomology of the Milnor fiber $F_V$ of the hypersurface $\{f=c\}$ at some point in $V$.

The integers $\mu_V$ appearing in \eqref{nvgi} can also be used to give a new interpretation for the number of Morse critical points on the regular part of $X$ in a Morsification of $f$. More precisely, 
one has the following result from \cite{MT}.
\bt{\rm \cite[Theorem 1.2]{MT}}\label{morsenumberi}
The number of Morse critical points on $X_{\reg}$ in a generic deformation $f_t:=f-t\ell\colon X \to \bC$ of $f$ is given by:
\be\label{mnui}
\# \Sing(f_t|_{X_{\reg}})=m_\infty + (-1)^{\dim X-1}  \sum_{c \in \bC} \left( \sum_{V \subset f^{-1}(c) \cap  \Sing_\mathscr{S} f} \chi(V \setminus V \cap H_t) \cdot \mu_V \right) ,
\ee
where $m_\infty $ is the number of points of $\Sing(f_t|_{X_{\reg}})$ that escape to infinity as $t \to 0$,  the first sum is over the critical values $c$ of $f$, and $H_t:=\ell^{-1}(t)$ is a generic hyperplane.
\et

The above formula is a direct application of \eqref{eq_maini} and Theorem \ref{edf}. When no critical points of $f_t|_{X_{\reg}}$ escape at infinity as $t\to 0$, one also obtains from \eqref{edi} and \eqref{mnui} a new and explicit formula for the ED degree.

\bex
Let $X=\bC^2$, with the stratification consisting of a single stratum. Consider $f(x,y)=x+x^2y$ and $\ell(x,y)=x+y$.  Then $f$ has no singularities in $\bC^2$ (though it has a singularity ``at infinity'', namely $p=[0:1:0]$). On the other hand, $f_t=f-t\ell$ has two Morse singularities. Formula \eqref{mnui} shows that these two Morse points escape to infinity as $t \to 0$ (and in fact it is easy to see that they converge to $p$, asymptotically to the fiber $f^{-1}(0)$).
\eex



\end{document}